%% file: main.tex
\documentclass[11pt, a4paper]{article}

\input{macros}

\newcommand{\stageC}{\varphi} 
\newcommand{\ExptCtgF}{V} 

\newcommand{\rp}{\bm{p}}

\newcommand{\ctr}{\mathrm{ctr}}

\renewcommand{\defeq}{=}

\newcommand{\raisemath}[2]{\raisebox{#1}{$#2$}}

\makeatletter
\def\@maketitle{%
  \newpage
  \begin{center}%
  \let \footnote \thanks
    {\LARGE \@title \par}%
    \vskip 1.5em%
    {\large
      \lineskip .5em%
      \begin{tabular}[t]{c}%
        \@author
      \end{tabular}\par}%
  \end{center}%
  }
\makeatother

\title{\textbf{Epi-Consistent Approximation\\of Stochastic Dynamic Programs}}
\author{\textsl{Dominic S. T. Keehan} \\
{\small University of Auckland} \\
{\small \texttt{dkee331@aucklanduni.ac.nz}}\\
\and \textsl{Johannes O. Royset} \\
{\small University of Southern California}\\
{\small \texttt{royset@usc.edu}}}
\date{\small 8\textsuperscript{th} of July, 2025}

\begin{document}

\maketitle

\noindent \textbf{Abstract:} We study the consistency of stochastic dynamic programs under converging probability distributions and other approximations. Utilizing results on the epi-convergence of expectation functions with varying measures and integrands, and the Attouch--Wets distance, we show that appropriate equi-semicontinuity assumptions assure epi-consistency. A number of examples illustrate the approach. In particular, we permit both unbounded and simultaneously approximated stage-cost functions, and treat an example with approximated constraints.

epi-convergence, Attouch--Wets distance.

\section{Introduction}
Usually in stochastic optimization problems the true probability distributions of the random variables involved must be approximated. It is important then that the decisions obtained from solving these approximating problems are statistically consistent. In this paper we consider stochastic optimal control problems where the solutions of associated stochastic dynamic programming equations are approximated. These equations involve a dependence between an objective function and a \emph{varying} function. Compared to optimizing the expectation of a \emph{fixed} function --- where consistency follows from suitably uniform versions of the strong law of large numbers (SLLN) \parencite{king-SLLN-consistency, artstein-SLLN-consistency} --- establishing consistency in this setting is more challenging.

Given a closed state space $\mathcal{X}\subset\reals^n$ and outcome space $\Xi\subset\reals^m$, we consider discrete-stage stochastic optimal control problems of the form
\begin{alignat}{3}\label{problem:sum}
& \kern0.04em \minimize_{y^1,\ldots,\,y^T} && \quad \Expt_{\Prob^T}\Biggl[\,\sum_{t=1}^{T}\beta^{t-1}\stageC(x^t,x^{t+1},\rxi^t) \Biggr]
&&\quad \tag{SOC}\\
& \subjectTo &&
\quad x^{t+1} = y^{t}(x^t,\rxi^t) \in \mathcal{Y}(x^t,\rxi^t) &&\quad \forall t\in[T]\defeq\{1,\ldots,T\},\notag
\end{alignat}
for initial state $x^1 \in\mathcal{X}$. The horizon $T$ may be finite or infinite, and the stagewise-independent random vector $\rxi^t$ has distribution $\Prob$ supported on $\Xi$. Here $\beta \in (0,1)$ appears with a power and $\stageC:\mathcal{X}\times\mathcal{X}\times\Xi\to\reals$. The set-valued mapping ${\mathcal{Y}:\mathcal{X}\times\Xi\tto\mathcal{X}}$, assumed throughout to be nonempty and compact valued, defines admissible policies $\{y^1,\ldots,y^T: \mathcal{X}\times\Xi\to\mathcal{X}\}$. Naturally, the problem requires additional assumptions to be well defined. We focus on the associated stochastic dynamic programming equations and will return to the discussion of assumptions in the following sections.

Stochastic optimal control problems are faithfully expressed through the solution of stochastic dynamic programming equations when the principle of optimality holds \parencite{light-principle-of-optimality-2024}. Establishing the consistency of an approximation to stochastic dynamic programming equations then entails: (i) showing that a solution to the true equations exists and satisfies the principle of optimality, and (ii) showing that this solution is approached by the solutions of the approximating equations. Typically, (i) is addressed with boundedness assumptions (see, e.g., \parencite[Chapters~3 and 4]{bertsekas1996stochastic} or \parencite[Chapter~9]{stokey-lucas:recursive-methods}), but more generally with monotone convergence \parencite[Chapter~5]{bertsekas1996stochastic} or inf-compactness \parencite{average-cost-unbounded-mdps-feinberg, discounted-cost-unbounded-mdps-feinberg}. In this work we address (ii) and provide results which ensure that the solution functions of the approximating stochastic dynamic programming equations epi-converge to solution functions of the true stochastic dynamic programming equations, and, as a consequence, that the approximating control policies for (\ref{problem:sum}) are consistent. 

There are existing works on the consistent approximation of multistage stochastic optimization problems. For finite-horizon problems, \textcite{pennanen-epiconvergent-discretisations, pennanen2009epi} establishes consistency by placing conditions on the scenario trees of the approximating control problem and requiring lower compactness of the objective function over the decision space. These conditions can be quite restrictive; the results may not hold when the support of the underlying distribution has holes \parencite[\S~5.1]{pennanen2009epi}, and lower compactness does not hold when the objective function cannot be lower bounded by a linear function \parencite[Theorem~5]{lower-compactness-ioffe} which excludes some problems involving saddle functions. \textcite{shapiro-multistage-consistency} also considers a related problem involving consistency in terms of an iterated (rather than joint) limit in the number of random samples in each stage. From a practical perspective this is unnatural as it requires a discretization where the number of samples in future stages must increase at a faster rate than the number of samples in prior stages, which may be difficult to implement. There is also an existing work that includes a consistency result for infinite-horizon problems. \textcite{shapiro-periodic-consistency} utilize the contraction-mapping property of Bellman operators defined on a complete metric space of bounded functions under the sup-norm. This approach fails when the stage-cost functions are unbounded. We avoid all of these limiting assumptions.

In this work we provide theoretical justification for algorithms which solve approximations of stochastic dynamic programs --- we {do not} deal with the practicalities of actually solving these problems computationally. However, stochastic dual dynamic programming--type algorithms have proven useful for this purpose, especially when the number of stages is not too large \parencite{pereira1991multi}. Infinite-horizon problems may also be approached by additionally approximating the number of stages \parencite{korf2006approximating} or with specialized algorithms; see, e.g., \parencite{Nannicini-infinite-horizon-alg, shapiro-periodic-consistency, dowson-policy-graph}. Although it is well known that both the algorithmic and statistical rate of convergence in long-horizon problems can be very slow, algorithms may be improved with specific bounding functions \parencite{shapiro2023dual-bounds} and statistical errors follow central limit--like behavior for certain problem classes \parencite{shapiro-central-limit}. Therefore, it is clear at least some stochastic dynamic programming problems can be feasibly solved to an acceptable accuracy.

The layout of this paper is as follows. In Section~\ref{section:preliminaries} we detail mathematical preliminaries and recall the necessary background on epi-convergence. Section~\ref{section:finite-horizon-stochastic-dynamic-programming} considers finite-horizon stochastic dynamic programming, presents intermediate results on the epi-convergence of varying sums and expectation functions, and provides conditions under which the solution functions of approximating stochastic dynamic programming equations epi-converge to solution functions of the true equations. Section~\ref{section:infinite-horizon-stochastic-dynamic-programming} considers infinite-horizon stochastic dynamic programming and adapts the results of the previous section to this setting using the metric structure imposed by the Attouch--Wets distance. We conclude in Section~\ref{section:examples} by considering a number of examples, including problems with unbounded stage-cost functions and simultaneously approximated constraints.

\section{Preliminaries}\label{section:preliminaries}
For the measurable space $(\Xi,\mathfrak{B}(\Xi))$ where $\mathfrak{B}(\Xi)$ is the Borel $\sigma$-algebra on $\Xi$, we write ${\Expt_{\Prob}\bigl[f(\rxi)\bigr] \defeq \int_{\Xi}f(\xi)\,\drv\Prob(\xi)}$ for the expectation of a Borel-measurable function $f:\Xi\to\reals$ with respect to a Borel probability distribution $\Prob$ when this is well defined. (We use boldface to distinguish random vectors from their outcomes.) We denote the weak convergence of a sequence of distributions $\{\Prob_{\nu}\}_{ \nu\in\nats}$ on $(\Xi,\mathfrak{B}(\Xi))$ to a distribution $\Prob$ on $(\Xi,\mathfrak{B}(\Xi))$ by $\Prob_{\nu}\weaklyto\Prob$; see \parencite{convergence-of-probability-measures:Billingsley} and \parencite[Chapter~13]{probability-theory:Klenke} for details on weak convergence.

Let $\overline{\reals}\defeq\reals\cup\{-\infty,\infty\}$. For a closed subset $\mathcal{Z}\subset\reals^{r}$, denote by $\fcns(\mathcal{Z})$ the set of $\overline{\reals}$-valued functions on $\mathcal{Z}$ excluding $f = \infty$. A function $f\in\fcns(\mathcal{Z})$ is \emph{proper} if $f(z)>-\infty$ for all $z\in\mathcal{Z}$. A sequence of functions $\{f_{\nu}\in\fcns(\mathcal{Z})\}_{{\nu}\in\nats}$ \emph{epi-converges} to a function $f\in\fcns(\mathcal{Z})$, denoted $f_{\nu} \epito f$, if for each $z\in\mathcal{Z}$ it holds that
    \begin{alignat*}{2}
        \forall z_{\nu}\in\mathcal{Z} &\to z, &\quad \liminf f_{\nu}(z_{\nu}) &\geq f(z)\\
        \exists z_{\nu}\in\mathcal{Z} &\to z, &\quad \limsup f_{\nu}(z_{\nu}) &\leq f(z);
    \end{alignat*}
these are the \emph{liminf} and \emph{limsup} portions of epi-convergence, respectively. The \emph{outer limit} of a sequence of sets, denoted by $\outerlimit$, is the collection of cluster points to which subsequences of elements within the sets converge. If $f_{\nu} \epito f$, then $\outerlimit (\argmin f_{\nu}) \subset \argmin f$, provided that $f$ is proper. In addition, if ${\cup_{\nu\in\nats} \{z\in\mathcal{Z}:f_{\nu}(z) < \infty \}}$ is bounded, then $\inf f_{\nu} \to \inf f$ as well; see, e.g., \parencite[\S~4.E]{royset:an-optimization-primer} and a slight extension of \parencite[Theorem~5.5(b), (d)]{royset:an-optimization-primer}. 

A function $f\in\fcns(\mathcal{Z})$ is \emph{lower semicontinuous} (lsc) if $\liminf f(z_{\nu})\geq f({z})$ for every ${z_{\nu}\in\mathcal{Z}\to z}$; denote by $\lscfcns(\mathcal{Z})$ the subset of $\fcns(\mathcal{Z})$ consisting of lsc functions. 
For $\bar{z}\in\mathcal{Z}$ and $\delta \geq 0$, let ${\ball_{{\delta}}(\bar{z})} \defeq {\{z\in\mathcal{Z} : \lVert z-\bar{z} \rVert_2\leq \delta\}}$, this being the closed ball of radius $\delta$ centered at $\bar{z}$ in $\mathcal{Z}$. A sequence of functions ${\{f_{\nu}\in\lscfcns(\mathcal{Z})\}_{{\nu}\in\nats}}$ is \emph{equi-lsc at} $\bar{z}\in\mathcal{Z}$ if either $\limsup f_{\nu}(\bar{z}) \to -\infty$, or for each $\rho,\varepsilon\in(0,\infty)$ there exists $\bar{\nu}\in\nats$ and $\delta>0$ such that
\begin{equation*}
    \inf_{z\in\ball_{\delta}(\bar{z})} f_{\nu}(z) \geq \min\{f_{\nu}(\bar{z})-\varepsilon,\rho\} \quad \forall  {\nu} \geq {\bar{\nu}}.
\end{equation*}
We say that ${\{f_{\nu}\}_{{\nu}\in\nats}}$ is \emph{equi-lsc} if it is equi-lsc at every $\bar{z}\in\mathcal{Z}$.

A sequence of functions $\{f_{\nu}\in\fcns(\mathcal{Z})\}_{{\nu}\in\nats}$ \emph{converges pointwise} to a function $f\in\fcns(\mathcal{Z})$, denoted $f_{\nu} \pointto f$, if $f_{\nu}(z)\to f(z)$ for each $z\in\mathcal{Z}$. If $\{f_{\nu}\}_{{\nu}\in\nats}$ is equi-lsc, then $f_{\nu} \pointto f$ if and only if $f_{\nu} \epito f$ \parencite[Theorem~7.10]{rockafellar-wets-variational-analysis}. We also need the notions of hypo-convergence ($\raisemath{-2.75pt}{\hypoto}$), upper semicontinuity (usc), and equi-upper semicontinuity (equi-usc); these have parallel definitions to those of epi-convergence, lower semicontinuity, and equi-lower semicontinuity, respectively; see \parencite{rockafellar-wets-variational-analysis} for details.

\section{Finite-Horizon Stochastic Dynamic Programming}\label{section:finite-horizon-stochastic-dynamic-programming}
In this section we consider the finite-horizon version of the problem (\ref{problem:sum}). A route to a solution is provided by finding real-valued functions $\{\ExptCtgF^t:\mathcal{X}\to\reals;\, t\in [T]\}$ satisfying the stochastic dynamic programming equations
\begin{equation}\label{equation:finite-horizon-value-function}
\ExptCtgF^t(x)=\Expt_{\Prob} \biggl[\,\inf_{y\in\mathcal{Y}(x,\rxi)}\Bigl\{\stageC(x,y,\rxi)+\beta\ExptCtgF^{t+1}(y)\Bigr\}\biggr] \quad \forall x\in\mathcal{X}, \; \forall t \in [T],
\end{equation}
where $\ExptCtgF^{T+1} = 0$. Given such $\{\ExptCtgF^t;\, t\in [T]\}$, a function  $y^{t}:\mathcal{X}\times\Xi\to\mathcal{X}$ satisfying 
\begin{equation*}
y^{t}(x,\xi)\in\argmin_{y\in\mathcal{Y}(x,\xi)}\Bigl\{\stageC(x,y,\xi)+\beta\ExptCtgF^{t+1}(y)\Bigr\} \quad\forall(x,\xi)\in\mathcal{X}\times\Xi
\end{equation*}
with $\xi\mapsto y^{t}(x,\xi)$ measurable for every $x\in\mathcal{X}$ defines a decision rule at stage $t$. Repeating this for each stage defines an admissible policy for (\ref{problem:sum}). If ${\{V^t;\, t\in [T]\}}$ satisfies the principle of optimality, then such a policy solves (\ref{problem:sum}) and the value $V^{1}(x^1)$ is equal to its optimal objective value; see \parencite[Propositions~3.1 and 3.2]{bertsekas1996stochastic} and \parencite[Theorem~9.2]{stokey-lucas:recursive-methods} for sufficient conditions.

For a sequence of probability distributions $\{\Prob_{\nu}\}_{\nu\in\nats}$ on $(\Xi,\mathfrak{B}(\Xi))$ {approximating $\Prob$} and a sequence of functions $\{\stageC_{\nu}:\mathcal{X}\times\mathcal{X}\times\Xi\to\reals\}_{\nu\in\nats}$ {approximating $\varphi$}, solutions to (\ref{equation:finite-horizon-value-function}) can be approximated by finding a sequence of real-valued functions $\{V^t_{\nu}:\mathcal{X}\to\reals;\, t\in[T]\}_{\nu\in\nats}$ satisfying the stochastic dynamic programming equations
\begin{equation}\label{equation:finite-horizon-saa-value-function}
    \ExptCtgF_{\nu}^{t}(x)=\Expt_{\Prob_{\nu}}\biggl[\,\inf_{y\in\mathcal{Y}(x,\rxi)}\Bigl\{\stageC_{\nu}(x,y,\rxi)+\beta \ExptCtgF_{\nu}^{t+1}(y)\Bigr\}\biggr] \quad\forall x\in\mathcal{X},\; \forall t\in[T],
\end{equation}
where $V^{T+1}_{\nu} = 0$. Given such $\{V_{\nu}^t;\,  t\in[T]\}$, an approximating decision rule $y^{t}_{\nu}:\mathcal{X}\times\Xi\to\mathcal{X}$ at stage $t$ is defined via
\begin{equation*}
y_{\nu}^{t}(x,\xi)\in\argmin_{y\in\mathcal{Y}(x,\xi)}\Bigl\{\stageC_{\nu}(x,y,\xi)+\beta\ExptCtgF^{t+1}_{\nu}(y)\Bigr\} \quad\forall(x,\xi)\in\mathcal{X}\times\Xi
\end{equation*}
and this therefore defines an approximating policy. Approximating the set-valued mapping $\mathcal{Y}$ can also be treated in the present setting through its dependence on values in the outcome space; see Subsection~\ref{subsection:inventory-control-autoregressive-randomness}.

The equations in (\ref{equation:finite-horizon-value-function}) feature a \emph{Bellman operator} $B:\lscfcns(\mathcal{X})\to\fcns(\mathcal{X})$ which for \newline${f \in \lscfcns(\mathcal{X})}$ at $x\in\mathcal{X}$ has value
\begin{equation}
    B(f)(x) \defeq \Expt_{\Prob}\biggl[\,\inf_{y\in\mathcal{Y}(x,\rxi)}\Bigl\{\stageC(x,y,\rxi)+\beta f(y)\Bigr\}\biggr],
\label{equation:bellman-operator}
\end{equation}
and they may be restated as $V^t = B(V^{t+1})$. It will be helpful to also consider the function $b:\lscfcns(\mathcal{X})\to\fcns(\mathcal{X}\times\Xi)$ which for $f \in\lscfcns(\mathcal{X})$ at $(x,\xi)\in\mathcal{X}\times\Xi$ has value
\begin{equation*}
    b(f)(x,\xi) \defeq \inf_{y\in\mathcal{Y}(x,\xi)}\Bigl\{\stageC(x,y,\xi)+\beta f(y)\Bigr\}.
\end{equation*}
By definition, $B(f) = \Expt_{\Prob}\bigl[b(f)(\blank,\rxi)\bigr]$. We similarly define $\{B_{\nu}:\lscfcns(\mathcal{X})\to\fcns(\mathcal{X});\, \nu\in\nats\}$ and $\{b_{\nu}:\lscfcns(\mathcal{X})\to\fcns(\mathcal{X}\times\Xi);\, \nu\in\nats\}$ for the approximating equations in (\ref{equation:finite-horizon-saa-value-function}). 

To ensure that the solutions functions to (\ref{equation:finite-horizon-value-function}) and (\ref{equation:finite-horizon-saa-value-function}) are well defined, we make the following assumption.
\begin{assumption}\label{assumption:finite-horizon-regularity}
    There exists a set $\mathcal{V}\subset\lscfcns(\mathcal{X})$ with $0 \in \mathcal{V}$ satisfying the following condition:
     \begin{enumerate}
     \item[] For each  $V\in\mathcal{V}$ and $x\in\mathcal{X}$ the function $b(V)(x,\blank)$ is integrable with respect to $\Prob$, and the Bellman operator $B$ maps $\mathcal{V}$ into itself. 
     \end{enumerate}
     Moreover, for each $\nu\in\nats$ the set $\mathcal{V}$ satisfies the equivalent condition for $b_{\nu}$, $\Prob_{\nu}$, and $B_{\nu}$.
\end{assumption}
\noindent Under Assumption~\ref{assumption:finite-horizon-regularity} the function $V^{t} = B(V^{t+1}) = \Expt_{\Prob}\bigl[b(V^{t+1})(\blank,\rxi)\bigr] \in \mathcal{V}$ is real valued for each $V^{t+1}\in\mathcal{V}$. It follows via induction that the solution functions to (\ref{equation:finite-horizon-value-function}) are well defined. 
Similar observations apply to (\ref{equation:finite-horizon-saa-value-function}). Section~\ref{section:examples} furnishes examples where the assumption holds.

We seek a sequence of approximating real-valued solution functions $\{V_{\nu}^t;\,  t\in[T]\}_{\nu\in\nats}$ to (\ref{equation:finite-horizon-saa-value-function}), consistent in the sense that the inclusion
\begin{equation*}
    \outerlimit \biggl(\,\argmin_{y\in\mathcal{Y}(x,\xi)}\Bigl\{\stageC_{\nu}(x,y,\xi)+\beta{V_{\nu}^{t+1}}(y)\Bigr\}\!\biggr)
\subset \argmin_{y\in\mathcal{Y}(x,\xi)}\Bigl\{\stageC(x,y,\xi)+\beta\ExptCtgF^{t+1}(y)\Bigr\}  \quad \forall (x,\xi)\in\mathcal{X}\times\Xi
\end{equation*}
holds for some real-valued solution functions $\{V^t;\,  t\in[T]\}$ to (\ref{equation:finite-horizon-value-function}).  In other words, each sequence of decisions obtained from the approximating problems may only converge (if at all) to an optimal decision for the true stochastic dynamic programming equations. Whether or not solutions to (\ref{equation:finite-horizon-saa-value-function}) satisfy the principle of optimality relative to an underlying approximation of (\ref{problem:sum}) is not strictly required --- the problem (\ref{problem:sum}) need not even be well defined. In this regard, throughout the paper we only assume that $\mathcal{X}$ and $\Xi$ are closed, and that $\mathcal{Y}$ is nonempty and compact valued.

The equations in (\ref{equation:finite-horizon-saa-value-function}) involve a sum of two functions that both vary with $\nu$. To ensure that epi-convergence is preserved under this operation, we prove the following lemma. Its conclusion is standard, although its proof is slightly nonstandard in that it applies to functions with domains that are strict subsets of finite-dimensional Euclidean spaces.
\begin{lemma}\label{lemma:epi-converging-sums}
For functions $\{f,f_{\nu},g,g_{\nu}:\mathcal{X}\to{\reals};\, \nu\in\nats\}$, if $f_{\nu}\epito f$, $f_{\nu}\pointto f$, $g_{\nu}\epito g$, and $g_{\nu}\pointto g$, then
\begin{equation*}
\inf_{y\in\mathcal{Y}(x,\xi)}\bigl\{f_{\nu}(y)+\beta g_{\nu}(y)\bigr\} \to \inf_{y\in\mathcal{Y}(x,\xi)}\bigl\{f(y)+\beta g(y)\bigr\} \quad \forall (x,\xi)\in\mathcal{X}\times\Xi.
\end{equation*}
\end{lemma}

\begin{proof}
Fix $(x,\xi)\in\mathcal{X}\times\Xi$. For each $\nu\in\nats$ define $\psi_{\nu}:\reals^{n}\to\overline{\reals}$ by $\psi_{\nu}(y) \defeq f_{\nu}(y)+\beta g_{\nu}(y)$ if $y\in\mathcal{Y}(x,\xi)$ and by $\psi_{\nu}(y) \defeq \infty$ otherwise. Define $\psi:\reals^{n}\to\overline{\reals}$ similarly. Following \parencite[Proposition~4.19(a)]{royset:an-optimization-primer}, we first show $\psi_{\nu} \epito \psi$. For $y\in\reals^{n}$, if ${y_{\nu}\in\reals^{n}\to y}$, then either $\liminf \psi_{\nu}(y_{\nu})=\infty \geq \psi(y)$ and the liminf portion of epi-convergence is satisfied, or there exists a subsequence with $y_{\nu_{\kappa}}\in\mathcal{Y}(x,\xi)\to y$ such that $\liminf \psi_{\nu}(y_{\nu})= \liminf \psi_{\nu_{\kappa}}(y_{\nu_{\kappa}})$. In the second case we construct a padded version of the sequence of indices $\{\nu_{\kappa}\}_{\kappa\in\nats}$ as follows: for each $\kappa\in\nats$ set ${{\nu}_{\kappa_{\ell}}} = \min_{i\in\nats}\{\nu_i:\nu_i\geq \kappa\}$. Now, $\{(\nu_{\kappa},\nu_{\kappa})\}_{\kappa\in\nats}$ is a subsequence of $\{(\kappa,{\nu_{\kappa_{\ell}}})\}_{\kappa\in\nats}$, which implies that
\begin{equation*}
    \liminf f_{\nu_{\kappa}}(y_{\nu_{\kappa}}) \geq \liminf f_{{\kappa}}(y_{{\nu_{\kappa_{\ell}}}}) \quad \text{and} \quad \liminf\beta g_{\nu_{\kappa}}(y_{\nu_{\kappa}}) \geq \liminf\beta g_{k}(y_{{\nu_{\kappa_{\ell}}}}).
\end{equation*} 
Clearly, $y_{{\nu_{\kappa_{\ell}}}} \in\mathcal{Y}(x,\xi)\to y$ as well, and since $f_{\nu} \epito f$ and $g_{\nu} \epito g$, we have that
\begin{equation*}
\liminf \beta f_{k}(y_{{\nu_{\kappa_{\ell}}}}) \geq \beta f(y) \quad \text{and} \quad \liminf \beta g_{k}(y_{{\nu_{\kappa_{\ell}}}}) \geq \beta g(y).    
\end{equation*} 
Together these yield
\begin{align*}
    \liminf \psi_{\nu}(y_{\nu}) &= \liminf \bigl(f_{\nu_{\kappa}}(y_{\nu_{\kappa}})+\beta g_{\nu_{\kappa}}(y_{\nu_{\kappa}})\bigr)\\
    & \geq \liminf f_{\nu_{\kappa}}(y_{\nu_{\kappa}})+ \liminf\beta g_{\nu_{\kappa}}(y_{\nu_{\kappa}}) \\
    & \geq \liminf f_{{\kappa}}(y_{{\nu_{\kappa_{\ell}}}})+ \liminf\beta g_{k}(y_{{\nu_{\kappa_{\ell}}}}) \\
    & \geq f(y)+ \beta g(y) \\
    &= \psi(y).
\end{align*}
The limsup portion of epi-convergence follows for each $y\in\reals^{n}$ when taking the sequence with $y_{\nu} = y$ for each $\nu\in\nats$; either $y\in\mathcal{Y}(x,\xi)$ and
\begin{align*}
    \limsup \psi_{\nu}(y_{\nu}) &= \limsup \bigl(f_{\nu}(y)+ \beta g_{\nu}(y)\bigr)\\
    & \leq \limsup f_{\nu}(y)+ \limsup \beta g_{\nu}(y) \\
    &= f(y)+ \beta g(y) \\
    &= \psi(y)
\end{align*}
where the second equality holds since $f_{\nu} \pointto f$ and $g_{\nu} \pointto g$, or $y\notin\mathcal{Y}(x,\xi)$ and 
\begin{equation*}
\limsup \psi_{\nu}(y_{\nu}) = \limsup \psi_{\nu}(y) = \infty = \psi(y).
\end{equation*}
We now establish the claimed convergence. The function $\psi$ is proper since it takes finite values on the nonempty set $\mathcal{Y}(x,\xi)$ and $\infty$ elsewhere. Moreover, ${\{y\in\reals^n : \psi_{\nu}(y) < \infty \}} = \mathcal{Y}(x,\xi)$ for each $\nu\in\nats$, and so ${\cup_{\nu\in\nats} \{y\in\reals^n : \psi_{\nu}(y) < \infty \}} = \mathcal{Y}(x,\xi)$. As assumed throughout, this set is compact and therefore bounded. Thus, \parencite[Theorem~5.5(d)]{royset:an-optimization-primer} applies to $\{\psi,\psi_{\nu};\, \nu\in\nats\}$ and
\begin{equation*}
	\inf_{y\in\mathcal{Y}(x,\xi)} \bigl\{f_{\nu}(y)+\beta g_{\nu}(y)\bigr\} = \inf \psi_{\nu} \to \inf \psi =  \inf_{y\in\mathcal{Y}(x,\xi)} \bigl\{f(y)+\beta g(y) \bigr\},
\end{equation*}
which proves the result.
\end{proof}

The solutions to the equations in (\ref{equation:finite-horizon-saa-value-function}) are expectation functions with a probability distribution and an integrand that vary with $\nu$. We use the following result of \textcite{feinberg-epiconvergence-varying-integrands} to establish epi-convergence for functions of this type. In our setting $\mathcal{X}\subset\reals^n$ and $\Xi\subset\reals^m$ are the closed state and outcome spaces. 

Let $\Indicator$ be the event indicator; that is $\Indicator \{\mathcal{E}\}=1$ if the event $\mathcal{E}$ is true and $\Indicator \{\mathcal{E}\}=0$ if it is false.
\begin{proposition}[{\parencite*{feinberg-epiconvergence-varying-integrands}, Theorem~3.7}]\label{proposition:epi-convergence-under-varying-measures-and-integrands}
For probability distributions $\{\Prob,\Prob_{\nu};\, \nu\in\nats\}$ on $(\Xi,\mathfrak{B}(\Xi))$ and functions $\{f,f_{\nu}:\mathcal{X}\times\Xi\to\overline{\reals};\, \nu\in\nats\}$, suppose that $\Prob_{\nu}\weaklyto\Prob$ and the following conditions are satisfied for each $\bar{x}\in\mathcal{X}$:
\begin{enumerate}[label={\normalfont(\roman*)}, ref={\roman*}]
    \item\label{proposition:epi-convergence-under-varying-measures-and-integrands-i} The functions $\{f(\bar{x},\blank),f_{\nu}(\bar{x},\blank);\,  \nu\in\nats\}$ are measurable.

\item\label{proposition:epi-convergence-under-varying-measures-and-integrands-ii} One has
    \begin{equation}\label{proposition:epi-convergence-under-varying-measures-and-integrands:equation:assymptotic-inf-integrable}
        \liminf_{\tau \to -\infty} \liminf_{(\nu,x) \to (\infty,\bar{x})} \Expt_{\Prob_{\nu}} \Bigl[f_{\nu}(x,\rxi)\Indicator\bigl\{ f_{\nu}(x,\rxi) \leq \tau \bigr\} \Bigr] = 0
        \end{equation}
        and for $\Prob$-almost every $\bar{\xi}\in\Xi$ it holds that
        \begin{equation}\label{proposition:epi-convergence-under-varying-measures-and-integrands:equation:liminf}
        \liminf_{(\nu,x,\xi)\to(\infty,\bar{x},\bar{\xi})} f_{\nu}(x,\xi) \geq f(\bar{x},\bar{\xi}).
        \end{equation}

    \item\label{proposition:epi-convergence-under-varying-measures-and-integrands-iii} There exists a sequence $x_{\nu}\in\mathcal{X}\to \bar{x}$ such that
    \begin{equation}\label{proposition:epi-convergence-under-varying-measures-and-integrands:equation:assymptotic-sup-integrable}
        \limsup_{\tau \to \infty} \limsup_{\nu \to \infty} \Expt_{\Prob_{\nu}} \Bigl[f_{\nu}(x_{\nu},\rxi)\Indicator\bigl\{  f_{\nu}(x_{\nu},\rxi) \geq \tau \bigr\} \Bigr] = 0
        \end{equation}
        and for $\Prob$-almost every $\bar{\xi}\in\Xi$ it holds that
        \begin{equation}\label{proposition:epi-convergence-under-varying-measures-and-integrands:equation:limsup}
            \limsup_{(\nu,\xi)\to(\infty,\bar{\xi})} f_{\nu}(x_{\nu},\xi) \leq f(\bar{x},\bar{\xi}).
        \end{equation}

\end{enumerate}
Then 
\begin{equation*}
    \Expt_{\Prob_{\nu}}\bigl[f_{\nu}(\blank,\rxi)\bigr] 
\epito \Expt_{\Prob}\bigl[ f(\blank,\rxi)\bigr].
\end{equation*}
\end{proposition}
We are now ready to prove our first main result.
\begin{theorem}\label{theorem:finite-horizon-consistency}
     Let Assumption~\ref{assumption:finite-horizon-regularity} hold and $\Prob_{\nu}\weaklyto\Prob$, $\varphi_{\nu}\epito\varphi$, and $\varphi_{\nu}\pointto\varphi$. Suppose that for each $t\in[T]$ the sequence of real-valued solution functions $\{\ExptCtgF_{\nu}^{t+1}\}_{\nu\in\nats}$ to (\ref{equation:finite-horizon-saa-value-function}) is equi-lsc and satisfies the following conditions for each $\bar{x}\in\mathcal{X}$:
     \begin{enumerate}[label={\normalfont(\roman*)}, ref={\roman*}]
    \item\label{theorem:finite-horizon-consistency-i} One has
     \begin{equation}\label{theorem:finite-horizon-consistency:equation:assymptotic-inf-integrable}
        \liminf_{\tau \to -\infty}
        \liminf_{(\nu,x)\to(\infty,\bar{x})} \Expt_{\Prob_{\nu}} \Bigl[b_{\nu}(\ExptCtgF_{\nu}^{t+1})(x,\rxi)\Indicator\bigl\{  b_{\nu}(\ExptCtgF^{t+1}_{\nu})(x,\rxi) \leq \tau \bigr\} \Bigr] = 0
     \end{equation}
     and that the sequence $\{b_{\nu}(\ExptCtgF_{\nu}^{t+1})\}_{\nu\in\nats}$ of lsc functions is equi-lsc.
     \item\label{theorem:finite-horizon-consistency-ii} One has
     \begin{equation}\label{theorem:finite-horizon-consistency:equation:assymptotic-sup-integrable}
        \limsup_{\tau \to \infty}
        \limsup_{\nu\to\infty} \Expt_{\Prob_{\nu}} \Bigl[b_{\nu}(\ExptCtgF_{\nu}^{t+1})(\bar{x},\rxi)\Indicator\bigl\{  b_{\nu}(\ExptCtgF^{t+1}_{\nu})(\bar{x},\rxi) \geq \tau \bigr\} \Bigr] = 0
     \end{equation}
     and that the sequence $\{b_{\nu}(\ExptCtgF_{\nu}^{t+1})(\bar{x},\blank)\}_{\nu\in\nats}$ of usc functions is equi-usc.
     \end{enumerate}
     Then, for each $t\in[T]$ it holds that $\ExptCtgF_{\nu}^{t+1} \epito \ExptCtgF^{t+1}$ and
\begin{multline*}
    \outerlimit \biggl(\,\argmin_{y\in\mathcal{Y}(x,\xi)}\Bigl\{\stageC_{\nu}(x,y,\xi)+\beta{V_{\nu}^{t+1}}(y)\Bigr\}\!\biggr)
\subset \argmin_{y\in\mathcal{Y}(x,\xi)}\Bigl\{\stageC(x,y,\xi)+\beta\ExptCtgF^{t+1}(y)\Bigr\} \neq \varnothing\\ \forall (x,\xi)\in\mathcal{X}\times\Xi.
\end{multline*}
\end{theorem}

\begin{proof}
    We proceed by induction in $t$ on the claim $\ExptCtgF_{\nu}^{t+1} \epito \ExptCtgF^{t+1}$. We have $\ExptCtgF_{\nu}^{T+1} = 0  \epito 0 = \ExptCtgF^{T+1}$ which is the base case. Assume $\ExptCtgF_{\nu}^{t+1} \epito \ExptCtgF^{t+1}$ for some $t\in[T]$. 
    With $\{\ExptCtgF_{\nu}^{t+1}\}_{\nu\in\nats}$ equi-lsc, it follows that $\ExptCtgF_{\nu}^{t+1} \pointto \ExptCtgF^{t+1}$ as well \parencite[Theorem~7.10]{rockafellar-wets-variational-analysis}. Since $\varphi_{\nu}\epito\varphi$ and $\varphi_{\nu}\pointto\varphi$, for every $(x,\xi)\in\mathcal{X}\times\Xi$ it holds that $\varphi_{\nu}(x,\blank,\xi) \epito \varphi(x,\blank,\xi)$ and $\varphi_{\nu}(x,\blank,\xi) \pointto \varphi(x,\blank,\xi)$, so Lemma~\ref{lemma:epi-converging-sums} applies to
\begin{align*}
    b_{\nu}(\ExptCtgF^{t+1}_{\nu})(x,\xi) &= \inf_{y\in\mathcal{Y}(x,\xi)}\Bigl\{\stageC_{\nu}(x,y,\xi)+\beta \ExptCtgF^{t+1}_{\nu}(y)\Bigr\}\\
    &\to \inf_{y\in\mathcal{Y}(x,\xi)}\Bigl\{\stageC(x,y,\xi)+\beta \ExptCtgF^{t+1}(y)\Bigr\} = b(\ExptCtgF^{t+1})(x,\xi),
\end{align*}
and we deduce that $b_{\nu}(\ExptCtgF^{t+1}_{\nu}) \pointto b(\ExptCtgF^{t+1})$. By (\ref{theorem:finite-horizon-consistency-i}) in the statement of the theorem, the sequence $\{b_{\nu}(V_{\nu}^{t+1})\}_{\nu\in\nats}$ is equi-lsc and it follows that $b_{\nu}(\ExptCtgF^{t+1}_{\nu}) \epito b(\ExptCtgF^{t+1})$ \parencite[Theorem~7.10]{rockafellar-wets-variational-analysis}. To conclude that ${\ExptCtgF_{\nu}^{t} \epito \ExptCtgF^{t}}$, we apply Proposition~\ref{proposition:epi-convergence-under-varying-measures-and-integrands} to the distributions $\{\Prob,\Prob_{\nu};\, \nu\in\nats\}$ and functions $\{b(V^{t+1}),b_{\nu}(V^{t+1}_{\nu});\, \nu\in\nats\}$, checking the conditions (\ref{proposition:epi-convergence-under-varying-measures-and-integrands-i})--(\ref{proposition:epi-convergence-under-varying-measures-and-integrands-iii}) as follows. 

(\ref{proposition:epi-convergence-under-varying-measures-and-integrands-i}): Satisfied due to the integrability (and therefore measurability) provided by Assumption~\ref{assumption:finite-horizon-regularity}. (\ref{proposition:epi-convergence-under-varying-measures-and-integrands-ii}): The asymptotic inf-integrability statement (\ref{proposition:epi-convergence-under-varying-measures-and-integrands:equation:assymptotic-inf-integrable}) is satisfied due to (\ref{theorem:finite-horizon-consistency-i}) in the statement of the theorem, and the liminf statement (\ref{proposition:epi-convergence-under-varying-measures-and-integrands:equation:liminf}) is satisfied since $b_{\nu}(V_{\nu}^{t+1}) \epito b(V^{t+1})$ implies that for every $(\bar{x},\bar{\xi})\in\mathcal{X}\times\Xi$,
\begin{gather*}
\forall (x_{\nu},\xi_{\nu}) \to (\bar{x},\bar{\xi}) \quad \liminf b_{\nu}(V_{\nu}^{t+1})(x_{\nu},\xi_{\nu}) \geq b(V^{t+1})(\bar{x},\bar{\xi})\\
\iff \liminf_{(\nu,x,\xi)\to(\infty,\bar{x},\bar{\xi})} b_{\nu}(V_{\nu}^{t+1})(x,\xi) \geq b(V^{t+1})(\bar{x},\bar{\xi}).
\end{gather*} 
(\ref{proposition:epi-convergence-under-varying-measures-and-integrands-iii}): For each $\bar{x}\in\mathcal{X}$, take the sequence with $x_{\nu} = \bar{x}$ for each $\nu\in\nats$. Due to (\ref{theorem:finite-horizon-consistency-ii}) in the statement of the theorem, the asymptotic sup-integrability statement (\ref{proposition:epi-convergence-under-varying-measures-and-integrands:equation:assymptotic-sup-integrable}) is satisfied. With ${b_{\nu}(V_{\nu}^{t+1}) \pointto b(V^{t+1})}$, we  have that $b_{\nu}(V_{\nu}^{t+1})(\bar{x},\blank) \pointto b(V^{t+1})(\bar{x},\blank)$, and again by (\ref{theorem:finite-horizon-consistency-ii}), the sequence $\{b_{\nu}(\ExptCtgF_{\nu}^{t+1})(\bar{x},\blank)\}_{\nu\in\nats}$ is equi-usc. Whence, $b_{\nu}(V_{\nu}^{t+1})(\bar{x},\blank) \overset{\text{h}}{\to} b(V^{t+1})(\bar{x},\blank)$ \parencite[Theorem~7.10]{rockafellar-wets-variational-analysis}. This implies that for every $\bar{\xi}\in\Xi$,
\begin{gather*}
\forall \xi_{\nu} \to \bar{\xi} \quad \limsup b_{\nu}(V_{\nu}^{t+1})(\bar{x},\xi_{\nu}) \leq b(V^{t+1})(\bar{x},\bar{\xi})\\
\iff \limsup_{(\nu,\xi)\to(\infty,\bar{\xi})} b_{\nu}(V_{\nu}^{t+1})(\bar{x},\xi) \leq b(V^{t+1})(\bar{x},\bar{\xi}),
\end{gather*}
which is the limsup statement (\ref{proposition:epi-convergence-under-varying-measures-and-integrands:equation:limsup}). Hence, Proposition~\ref{proposition:epi-convergence-under-varying-measures-and-integrands} applies, and we conclude that
\begin{equation*}
    \ExptCtgF_{\nu}^{t} = B_{\nu}(V_{\nu}^{t+1}) = \Expt_{\Prob_{\nu}}\bigl[b_{\nu}(V_{\nu}^{t+1})(\blank,\rxi)\bigr] \epito \Expt_{\Prob}\bigl[b(V^{t+1})(\blank,\rxi)\bigr] = B(V^{t+1}) =  \ExptCtgF^{t}.
\end{equation*}
 Reasoning similar to that in the proof of Lemma~\ref{lemma:epi-converging-sums} as well as \parencite[Theorem~5.5(b)]{royset:an-optimization-primer} then shows that the inclusion
\begin{multline*}
    \outerlimit \biggl(\,\argmin_{y\in\mathcal{Y}(x,\xi)}\Bigl\{\stageC_{\nu}(x,y,\xi)+\beta{V_{\nu}^{t+1}}(y)\Bigr\}\!\biggr)
\subset \argmin_{y\in\mathcal{Y}(x,\xi)}\Bigl\{\stageC(x,y,\xi)+\beta\ExptCtgF^{t+1}(y)\Bigr\} \neq \varnothing\\ \forall (x,\xi)\in\mathcal{X}\times\Xi
\end{multline*}
holds. To see that this $\argmin$ is nonempty, observe that $\varphi(x,\blank,\xi) + \beta V^{t+1}(\blank)$ is lsc and real valued, and that $\mathcal{Y}(x,\xi)$ is nonempty and compact by assumption. This proves the result.
\end{proof}

The statements (\ref{theorem:finite-horizon-consistency:equation:assymptotic-inf-integrable}) and (\ref{theorem:finite-horizon-consistency:equation:assymptotic-sup-integrable}) in the conditions~(\ref{theorem:finite-horizon-consistency-i}) and (\ref{theorem:finite-horizon-consistency-ii}) of Theorem~\ref{theorem:finite-horizon-consistency} are asymptotic semi-integrability statements that hold, for instance, under asymptotic local boundedness. Let $\rxi_1,\rxi_2,\ldots$ be random vectors which are independent and identically distributed according to $\Prob$ and suppose that $\{\xi_1,\xi_2,\ldots\}$ is an outcome of this random sequence.  We write $\pmProb_{\xi}$ for the point-mass probability distribution that assigns probability $1$ to $\xi\in\Xi$, and set $\Prob_{\nu} = \frac{1}{\nu}\sum_{i=1}^{\nu}\pmProb_{{\xi}_i}$ for each $\nu\in\nats$. Note that $\Prob_{\nu} \weaklyto \Prob$ for almost every outcome $\{{\xi}_1,{\xi}_2,\ldots\}$; see, e.g., \parencite[Problem~3.1]{convergence-of-probability-measures:Billingsley}. For a given $\bar{x}\in\mathcal{X}$, assume that there exists an integrable function $f:\Xi\to\reals$ and a constant $\delta > 0$ such that $f(\xi) \leq b_{\nu}(V_{\nu}^{t+1})(x,\xi)$ holds for all $(x,\xi) \in (\ball_{\delta}(\bar{x})\cap\mathcal{X})\times\Xi$ and $\nu\in\nats$ sufficiently large. Then
\begin{align*}
0 &\geq \liminf_{\tau \to -\infty} \liminf_{(\nu,x) \to (\infty,\bar{x})} \Expt_{\Prob_{\nu}} \Bigl[b_{\nu}(V_{\nu}^{t+1})(x,\rxi)\Indicator\bigl\{  b_{\nu}(V_{\nu}^{t+1})(x,\rxi) \leq \tau \bigr\} \Bigr] \\
 &\geq \liminf_{\tau \to -\infty} \liminf_{\nu \to \infty} \Expt_{\Prob_{\nu}} \Bigl[f(\rxi)\Indicator\bigl\{ f(\rxi) \leq \tau \bigr\} \Bigr]\\
 &= \liminf_{\tau \to -\infty}\Expt_{\Prob} \Bigl[f(\rxi)\Indicator\bigl\{  f(\rxi) \leq \tau \bigr\} \Bigr]\\
  &=  0,
\end{align*}
where the first equality follows for almost every outcome $\{{\xi}_1,{\xi}_2,\ldots\}$ due to the SLLN, and the second equality holds since $\Expt_{\Prob}\bigl[f(\rxi)\bigr]$ is finite. Moreover, the equi-semicontinuity conditions in Theorem~\ref{theorem:finite-horizon-consistency} can often be shown to hold using arguments involving convexity and concavity; see, e.g., \parencite{royset-wets-constrained-m-estimators}. In particular, it is sufficient for the functions to be locally Lipschitz continuous with common moduli. More examples follow in Section~\ref{section:examples}.

\section{Infinite-Horizon Stochastic Dynamic Programming}\label{section:infinite-horizon-stochastic-dynamic-programming}
In this section we consider the infinite-horizon version of the problem (\ref{problem:sum}). A route to a solution is provided by finding a real-valued function ${\ExptCtgF:\mathcal{X}\to\reals}$ satisfying the stochastic dynamic programming equation
\begin{equation}\label{equation:infinite-horizon-value-function}
     \ExptCtgF(x)=\Expt_{\Prob} \biggl[ \,\inf_{y\in\mathcal{Y}(x,\rxi)}\Bigl\{\stageC(x,y,\rxi)+\beta \ExptCtgF(y)\Bigr\}\biggr] \quad\forall x\in\mathcal{X}. 
\end{equation}
Solutions to (\ref{equation:infinite-horizon-value-function}) can be approximated by finding a real-valued function $V_{\nu}:\mathcal{X}\to\reals$ satisfying the approximating stochastic dynamic programming equation
\begin{equation}\label{equation:infinite-horizon-saa-value-function}
    \ExptCtgF_{\nu}(x)=\Expt_{\Prob_{\nu}}\biggl[\,\inf_{y\in\mathcal{Y}(x,\rxi)}\Bigl\{\stageC_{\nu}(x,y,\rxi)+\beta \ExptCtgF_{\nu}(y)\Bigr\}\biggr] \quad\forall x\in\mathcal{X}.
\end{equation}

In contrast to the finite-horizon stochastic dynamic programming equations, (\ref{equation:infinite-horizon-value-function}) and (\ref{equation:infinite-horizon-saa-value-function}) are functional fixed-point equations in their associated Bellman operators; they may be restated as $\ExptCtgF=B(\ExptCtgF)$ and $\ExptCtgF_{\nu}=B_{\nu}(\ExptCtgF_{\nu})$, respectively. For a set $\mathcal{V}\subset\lscfcns(\mathcal{X})$, even if $B(V)$ is a well-defined function for each $V\in\mathcal{V}$, the operator $B$ may not have a fixed point in $\mathcal{V}$. For now, given a set $\mathcal{V}\subset\lscfcns(\mathcal{X})$, we make the following integrability assumption.
\begin{assumption}\label{assumption:infinite-horizon-regularity}
    The set $\mathcal{V}\subset\lscfcns(\mathcal{X})$ satisfies the following condition:
     \begin{enumerate}
     \item[] For each $V\in\mathcal{V}$ and $x\in\mathcal{X}$ the function $b(V)(x,\blank)$ is integrable with respect to the probability distribution $\Prob$, and the Bellman operator $B$ maps $\mathcal{V}$ into itself.
     \end{enumerate}
    Moreover, for each $\nu\in\nats$ the set $\mathcal{V}$ satisfies the equivalent condition for $b_{\nu}$, $\Prob_{\nu}$, and $B_{\nu}$.
\end{assumption}
\noindent Assumption~\ref{assumption:infinite-horizon-regularity} differs from Assumption~\ref{assumption:finite-horizon-regularity} in that the existence of such a set $\mathcal{V}$ is not presupposed. For results on fixed-point problems it is natural to delegate verifying the existence of fixed points to the user of said results, as the analysis often requires utilizing context-dependent properties of the associated fixed-point operators. Additionally, the inclusion $0 \in \mathcal{V}$ is not present as the terminal condition $V^{T+1} = 0$ is no longer relevant in the infinite-horizon setting.

To study the asymptotic behavior of sequences of solutions to functional fixed-point problems, we metrize epi-convergence. Recall that a function $f\in\lscfcns(\mathcal{X})$ has \emph{epigraph} $\epi f \defeq {\{(x,\tau)\in\mathcal{X}\times\reals : f(x)\leq \tau\}}$. For a point $\bar{z}\in\mathcal{X}\times\reals$ and a set $\mathcal{Z}\subset\mathcal{X}\times\reals$, let $\dist(\bar{z},\mathcal{Z})\defeq \inf_{z\in\mathcal{Z}}\lVert z-\bar{z}\rVert_2 $, this being the point-to-set distance. Also, let $z_{\ctr} \in\mathcal{X}\times\reals$. The \emph{Attouch--Wets distance} \parencite{AttouchWets1986, AttouchWets1991, AttouchLucchettiWets1991} between two functions $f, g \in \lscfcns(\mathcal{X})$ is
\begin{equation*}
    \awDistance(f,g) \defeq \int_{0}^{\infty}\max_{z\in\ball_{\rho}(z_{\ctr})}\bigl\lvert\dist(z, \epi f)-\dist(z, \epi g) \bigr\rvert \exp({-\rho}) \,\drv\rho
\end{equation*}
and this defines a metric on $\lscfcns(\mathcal{X})$; for $\{f, f_{\nu};\, \nu\in\nats\}\subset{\lscfcns(\mathcal{X})}$, it holds that $f_{\nu} \epito f$ if and only if $\awDistance(f_{\nu},f) \to 0$. Moreover, under $\awDistance$, closed and bounded subsets of $\lscfcns(\mathcal{X})$ are compact \parencite[Proposition~4.45, Theorem~7.58]{rockafellar-wets-variational-analysis}. Closed subsets include those with elements that are convex, concave, pointwise bounded, or locally Lipschitz-continuous functions; see, e.g., \parencite[\S~4]{royset-wets-constrained-m-estimators}. The boundedness of subsets can be established by appealing to the inequality $\awDistance(f,g) \leq 1+\max\{\dist(z_{\ctr},\epi f), \dist(z_{\ctr},\epi g)\}$ \parencite[Proposition~3.1]{royset-approximations-2018}. For instance, if $z_{\ctr} \in \epi f\cap \epi g$, then $\max\{\dist(z_{\ctr},\epi f), \dist(z_{\ctr},\epi g)\}=0$ and $\awDistance(f,g) \leq 1$, so it suffices to show that all of the epigraphs share a common point. The choice of $z_{\ctr}$ alters the numerical value of $\awDistance$, but not the resulting topology.

For fixed-point functions $\{V=B(V), V_{\nu}=B_{\nu}(V_{\nu});\:  \nu\in\nats\}$, it holds that $V_{\nu}\epito V$ if and only if $B_{\nu}(V_{\nu})\epito B(V)$. This motivates the following continuous epi-convergence result in which we utilize the metric properties of $\awDistance$. Note that the expression $V=B(V)$ is understood to mean that $\awDistance(V,B(V))=0$. Since $\awDistance$ is a metric over $\lscfcns(\mathcal{X})$, with $V, B(V) \in \lscfcns(\mathcal{X})$ the functions $V$ and $B(V)$ agree pointwise which is concordant with the definition (\ref{equation:infinite-horizon-value-function}).
\begin{lemma}\label{lemma:continuous-convergence-of-fixed-point-operators}
Let Assumption~\ref{assumption:infinite-horizon-regularity} hold and suppose that $\mathcal{V}\subset\lscfcns(\mathcal{X})$ is a closed set under $\awDistance$. If $B_{\nu}(V_{\nu})\epito B(V)$ whenever ${V_{\nu}\in\mathcal{V}\epito V}$, then  
\begin{equation*}
	\outerlimit\{V\in\mathcal{V}:V=B_{\nu}(V)\} \subset \{V\in\mathcal{V}:V=B(V)\}.
\end{equation*}
In addition, if $\mathcal{V}$ is compact under $\awDistance$ and each $B_{\nu}$ has a fixed point $V_{\nu} = B_{\nu}(V_{\nu})$, then every cluster point of $\{V_{\nu}\}_{\nu \in \nats}$ is contained in $\{V\in\mathcal{V}:V=B(V)\}$ and this set is nonempty.
\end{lemma}

\begin{proof}
We view $\mathcal{V}\subset\lscfcns(\mathcal{X})$ as a metric space under $\awDistance$, and this metrizes epi-convergence. Let $V_{{\nu}_{\kappa}} \epito \bar{V}\in\mathcal{V}$ be a convergent subsequence of elements within the sequence of sets $\{\{V\in\mathcal{V}:V=B_{\nu}(V)\}\}_{{\nu}\in\nats}$. If no such subsequence exists, then the claim is vacuously true. We construct a padded version of the sequence of indices $\{\nu_{\kappa}\}_{\kappa\in\nats}$ as follows: for each $\kappa\in\nats$, set ${\nu}_{\kappa_{\ell}} = \min_{i\in\nats}\{\nu_i : \nu_i\geq \kappa\}$. Clearly, $V_{{\nu}_{\kappa_{\ell}}} \in \mathcal{V} \epito \bar{V}$, and it follows that $B_{\kappa}(V_{{\nu}_{\kappa_{\ell}}}) \epito B(\bar{V})$. Observe that $\{(\nu_{\kappa},\nu_{\kappa})\}_{\kappa\in\nats}$ is a subsequence of $\{(\kappa,{\nu_{\kappa_{\ell}}})\}_{\kappa\in\nats}$. But this means that $B_{{\nu}_{\kappa}}(V_{{\nu}_{\kappa}}) = V_{{{\nu}_{\kappa}}} \epito \bar{V}$ has the same limit as $B_{\kappa}(V_{{\nu}_{\kappa_{\ell}}}) \epito B(\bar{V})$. Thus, $\bar{V} = B(\bar{V})$, and the cluster point $\bar{V}$ is a fixed point of $B$. Since the subsequence $\{V_{{\nu}_{\kappa}}\}_{\nu\in\nats}$ was arbitrary, we have shown that 
\begin{equation*}
	\outerlimit\{V\in\mathcal{V}:V=B_{\nu}(V)\} \subset \{V\in\mathcal{V}:V=B(V)\}.
\end{equation*}
If in addition $\mathcal{V}$ is compact under $\awDistance$ and each $B_{\nu}$ has a fixed point $V_{\nu} = B_{\nu}(V_{\nu})$, then a convergent subsequence $V_{{\nu}_{\kappa}} \epito \bar{V}\in\mathcal{V}$ of $\{V_{\nu}\}_{{\nu}\in\nats}$ exists. Identical reasoning to that above then shows that the cluster point $\bar{V}$ is a fixed point of $B$. 
\end{proof}

We are now ready to prove our second main result.
\begin{theorem}\label{theorem:infinite-horizon-consistency}
Let Assumption~\ref{assumption:infinite-horizon-regularity} hold and $\Prob_{\nu}\weaklyto\Prob$, $\varphi_{\nu}\epito\varphi$, and $\varphi_{\nu}\pointto\varphi$. Suppose that $\mathcal{V}\subset\lscfcns(\mathcal{X})$ is a closed set under $\awDistance$. Further, suppose that every epi-convergent sequence $\{V_{\nu}\in\mathcal{V}\}_{\nu\in\nats}$ is equi-lsc and satisfies the following conditions for each $\bar{x}\in\mathcal{X}$:
     \begin{enumerate}[label={\normalfont(\roman*)}, ref={\roman*}]
    \item\label{theorem:infinite-horizon-consistency-i} One has
     \begin{equation}\label{theorem:infinite-horizon-consistency:equation:assymptotic-inf-integrable}
        \liminf_{\tau \to -\infty}
        \liminf_{(\nu,x)\to(\infty,\bar{x})} \Expt_{\Prob_{\nu}} \Bigl[b_{\nu}(\ExptCtgF_{\nu})(x,\rxi)\Indicator\bigl\{  b_{\nu}(\ExptCtgF_{\nu})(x,\rxi) \leq \tau \bigr\} \Bigr] = 0
     \end{equation}
     and that the sequence $\{b_{\nu}(\ExptCtgF_{\nu})\}_{\nu\in\nats}$ of lsc functions is equi-lsc.
     \item\label{theorem:infinite-horizon-consistency-ii}  One has
     \begin{equation}\label{theorem:infinite-horizon-consistency:equation:assymptotic-sup-integrable}
        \limsup_{\tau \to \infty}\,
        \limsup_{\nu\to\infty} \Expt_{\Prob_{\nu}} \Bigl[b_{\nu}(\ExptCtgF_{\nu})(\bar{x},\rxi)\Indicator\bigl\{  b_{\nu}(\ExptCtgF_{\nu})(\bar{x},\rxi) \geq \tau \bigr\} \Bigr] = 0
     \end{equation}
     and that the sequence $\{b_{\nu}(\ExptCtgF_{\nu})(\bar{x},\blank)\}_{\nu\in\nats}$ of usc functions is equi-usc.
     \end{enumerate}
Then,
\begin{equation*}
\outerlimit\{V\in\mathcal{V}:V=B_{\nu}(V)\} \subset \{V\in\mathcal{V}:V=B(V)\}.
\end{equation*}
In addition, if $\mathcal{V}$ is compact under $\awDistance$ and each $B_{\nu}$ has a fixed point $V_{\nu}=B_{\nu}(V_{\nu})$, then every cluster point of $\{V_{\nu}\}_{\nu \in \nats}$ is contained in $\{V\in\mathcal{V}:V=B(V)\}$ and this set is nonempty.
\end{theorem}

\begin{proof}
Let $V_{\nu}\in\mathcal{V}\epito V$. Following the proof of Theorem~\ref{theorem:finite-horizon-consistency}, we deduce that $b_{\nu}(\ExptCtgF_{\nu}) \pointto b(\ExptCtgF)$ and $b_{\nu}(\ExptCtgF_{\nu}) \epito b(\ExptCtgF)$, and together with (\ref{theorem:infinite-horizon-consistency-i}) and (\ref{theorem:infinite-horizon-consistency-ii}), that Proposition~\ref{proposition:epi-convergence-under-varying-measures-and-integrands} applies to $\{\Prob,\Prob_{\nu};\, \nu\in\nats\}$ and $\{b(V),b_{\nu}(V_{\nu});\, \nu\in\nats\}$. Thus,
\begin{equation*}
    B_{\nu}(V_{\nu}) = \Expt_{\Prob_{\nu}}\big[b_{\nu}(V_{\nu})(\blank,\rxi)\big] \epito \Expt_{\Prob}\big[b(V)(\blank,\rxi)\big] = B(V).
\end{equation*}
Hence, Lemma~\ref{lemma:continuous-convergence-of-fixed-point-operators} applies, and this proves the result.
\end{proof}

While Theorem~\ref{theorem:infinite-horizon-consistency} can be used to conclude that the cluster points of sequences of solution functions to the approximating fixed-point problems (\ref{equation:infinite-horizon-saa-value-function}) solve the true fixed-point problem (\ref{equation:infinite-horizon-value-function}), making statements about sequences of approximating decision rules obtained from the approximating solution functions is more delicate. We summarize the possibilities in the following proposition.
\begin{proposition}\label{proposition:consequences-of-non-uniqueness}
Let the conditions in the statement of Theorem~\ref{theorem:infinite-horizon-consistency} hold, including $\mathcal{V}$ being compact under $\awDistance$ and each $B_{\nu}$ having a fixed point $V_{\nu}=B_{\nu}(V_{\nu})$. Then the following are true:
    \begin{enumerate}[label={\normalfont(\roman*)}, ref={\roman*}]
    \item\label{proposition:consequences-of-non-uniqueness-i} If $V_{\nu_{\kappa}} \epito V$ is an epi-convergent subsequence of $\{V_{\nu}=B_{\nu}(V_{\nu})\}_{\nu\in\nats}$, then $V=B(V)$ and
    \begin{multline*} 
    \outerlimit \biggl(\,\argmin_{y\in\mathcal{Y}(x,\xi)}\Bigl\{\stageC_{\nu_{\kappa}}(x,y,\xi)+\beta V_{\nu_{\kappa}}(y)\Bigr\}\!\biggr) \subset \argmin_{y\in\mathcal{Y}(x,\xi)}\Bigl\{\stageC(x,y,\xi)+\beta V(y)\Bigr\}\neq\varnothing\\
    \forall(x,\xi)\in\mathcal{X}\times\Xi.
    \end{multline*} 
    \item\label{proposition:consequences-of-non-uniqueness-ii} If the fixed point $V=B(V)$ is unique, then
    \begin{multline*} 
    \outerlimit \biggl(\,\argmin_{y\in\mathcal{Y}(x,\xi)}\Bigl\{\stageC_{\nu}(x,y,\xi)+\beta V_{\nu}(y)\Bigr\}\!\biggr)\subset\argmin_{y\in\mathcal{Y}(x,\xi)}\Bigl\{\stageC(x,y,\xi)+\beta V(y)\Bigr\}\neq\varnothing\\
    \forall(x,\xi)\in\mathcal{X}\times\Xi.
     \end{multline*}
\end{enumerate}
\end{proposition}

\begin{proof}
    (\ref{proposition:consequences-of-non-uniqueness-i}): If $V_{\nu_{\kappa}} \epito V$, from Theorem~\ref{theorem:infinite-horizon-consistency} we deduce that the cluster point $V$ must satisfy $V=B(V)$. With $\{V_{\nu_{\kappa}}\}_{\kappa\in\nats}$ equi-lsc, it follows that $V_{\nu_{\kappa}} \overset{p}{\to} V$ as well \parencite[Theorem~7.10]{rockafellar-wets-variational-analysis}. Since $\varphi_{\nu}\epito\varphi$ and $\varphi_{\nu}\pointto\varphi$, for every $(x,\xi)\in\mathcal{X}\times\Xi$ it holds that $\varphi_{\nu}(x,\blank,\xi) \epito \varphi(x,\blank,\xi)$ and $\varphi_{\nu}(x,\blank,\xi) \pointto \varphi(x,\blank,\xi)$. Reasoning similar to that in the proof of Lemma~\ref{lemma:epi-converging-sums} as well as \parencite[Theorem~5.5(b)]{royset:an-optimization-primer} then shows that the claimed inclusion holds. To see that the $\argmin$ is nonempty, observe that $\varphi(x,\blank,\xi) + \beta V(\blank)$ is lsc and real valued, and that $\mathcal{Y}(x,\xi)$ is nonempty and compact (as assumed throughout).
    
    (\ref{proposition:consequences-of-non-uniqueness-ii}): If $V=B(V)$ is unique, then $V_{\nu_{\kappa}} \epito V$ for each epi-convergent subsequence of $\{V_{\nu}\}_{\nu\in\nats}$. Viewing  $\mathcal{V}\subset\lscfcns(\mathcal{X})$ as a compact metric space under $\awDistance$, it follows that $V_{\nu} \epito V$. Similar reasoning as above then shows that the claimed inclusion  and nonemptyness holds.
\end{proof}

\section{Examples}\label{section:examples}
In this section we demonstrate how the results developed in this paper can be applied. While we only consider the consistent approximation of infinite-horizon stochastic optimal control problems via Theorem~\ref{theorem:infinite-horizon-consistency}, similar arguments hold for finite-horizon problems and Theorem~\ref{theorem:finite-horizon-consistency}. In Subsection~\ref{subsection:inventory-control-example} we approximate a one-dimensional infinite-horizon problem and show that pointwise bounds on the solution to the approximating fixed-point problems (\ref{equation:infinite-horizon-saa-value-function}) can be derived from problem primitives. This makes it possible to specify a set $\mathcal{V}\subset\lscfcns(\mathcal{X})$ containing the approximating solutions for which we check that Theorem~\ref{theorem:infinite-horizon-consistency} applies. Next, in Subsection~\ref{subsection:inventory-control-inconsistency}, we consider an alternate version of the previous problem for which Theorem~\ref{theorem:infinite-horizon-consistency} does not apply. We show that the resulting approximating decision rules fail to converge to an optimal policy for the true control problem. Lastly, in Subsection~\ref{subsection:inventory-control-autoregressive-randomness} we approximate an infinite-horizon problem involving an autoregressive random variable, for which we adapt the approach taken in Subsection~\ref{subsection:inventory-control-example}.

Throughout this section we work with saddle functions. For two convex sets $\mathcal{Z}_1$ and $\mathcal{Z}_2$, a \emph{saddle function} on $\mathcal{Z}_1\times\mathcal{Z}_2$ is a real-valued function which is convex in its first argument and concave in its second argument, i.e., $f:\mathcal{Z}_1\times\mathcal{Z}_2\to\reals$ such that $z_1 \mapsto f(z_1,z_2)$ is convex for each $z_2\in\mathcal{Z}_2$ and $z_2\mapsto f(z_1,z_2)$ is concave for each $z_1\in\mathcal{Z}_1$.

\subsection{Revenue Optimization With Stochastic Prices}
\label{subsection:inventory-control-example}
For initial state $x^1\in\reals_+$, consider the infinite-horizon stochastic optimal control problem
\begin{alignat}{3}\label{problem:revenue-optimization}
& \kern0.04em \minimize_{y^1,y^2, \ldots } && \quad \Expt_{\Prob^{\infty}}\Biggl[\, \sum_{t=1}^{\infty}\beta^{t-1}\bigl(C(x^{t+1})-\rp^t(x^t-x^{t+1})\bigr) \Biggr]&&\quad\tag{ROSP}\\
& \subjectTo &&
\quad x^{t+1} = y(x^t,\rp^t)\in[0,x^t] \quad \forall t\in \nats, \notag
\end{alignat}
which is studied in \parencite{Keehan-et-al-MPC}. Here $x^t\in\reals_+$ is an inventory of some product, $\rp^t \sim \Prob$ is a random variable supported on $\reals_+$ representing the product's per-unit market price, and $C:\reals_+\to \reals_+$ is a continuous convex function giving the cost of storing inventory between stages. The decision to be made is for what inventory level to sell down to at the current realization of the price. 

The associated fixed-point problem is to find ${V}:\reals_+\to{\reals}$ such that
\begin{equation*}
    {V}(x)=\Expt_{\Prob}\biggl[\,\inf_{0\leq y\leq x}\Bigl\{C(y)-\rp(x-y)+\beta {V}(y)\Bigr\}\biggr] \quad \forall x\in\reals_+.
\end{equation*}
With $\Prob$ supported on all of $\reals_+$, the stage-cost function is unbounded. To approximate this problem, let $\rp_1,\rp_2,\ldots$ be random variables which are independent and identically distributed according to $\Prob$, and suppose that $\{p_1,p_2,\ldots\}$ is an outcome of this random sequence. We set $\Prob_{\nu} = \frac{1}{\nu}\sum_{i=1}^{\nu}\pmProb_{{p}_i}$ for each $\nu\in\nats$. 
Recall $\Prob_{\nu} \weaklyto \Prob$ for almost every such outcome $\{{p}_1,{p}_2,\ldots\}$. The approximating fixed-point problem is to find  $\ExptCtgF_{\nu}:\reals_+\to{\reals}$ such that
\begin{equation*}
    \ExptCtgF_{\nu}(x)=\Expt_{\Prob_{\nu}}\biggl[\,\inf_{0\leq y \leq x}\Bigl\{C(y)-\rp(x-y)+\beta \ExptCtgF_{\nu}(y)\Bigr\}\biggr] \quad \forall x\in\reals_+.
\end{equation*}
Since $\Prob_{\nu}$ has compact support and future states are constrained to a compact set, an argument utilizing \parencite[Theorems~9.6 and 9.8]{stokey-lucas:recursive-methods} shows that a unique continuous and convex solution function $\ExptCtgF_{\nu}$ satisfying the principle of optimality relative to the underlying approximation of (\ref{problem:revenue-optimization}) exists. 

In order to construct a set $\mathcal{V}\subset\lscfcns(\reals_+)$ for which the conditions of Theorem~\ref{theorem:infinite-horizon-consistency} hold, we derive the following pointwise bound on the approximating fixed-point functions $V_{\nu}$.
\begin{proposition}\label{proposition:revenue-optimization-value-function-bounds}
In the context of this subsection, suppose that $\Expt_{\Prob}[\rxi]$ is finite. Then, for each $\nu\in\nats$ it holds that
\begin{equation*}
    \inf_{\kappa\in\nats}-(1-\beta)^{-1} \Expt_{\Prob_{\kappa}}[\rp]x \leq \ExptCtgF_{\nu}(x) \leq (1-\beta)^{-1} C(x) \quad \forall x\in\reals_+,
\end{equation*}
and these bounds are finite valued for almost every outcome $\{{p}_1,{p}_2,\ldots\}$.
\end{proposition}

\begin{proof}
For a given $\nu\in\nats$, consider the underlying approximation of (\ref{problem:revenue-optimization}) induced by $\Prob_{\nu}$:
\begin{alignat}{3}
& \kern0.04em \minimize_{y^1,y^2,\ldots } && \quad \Expt_{\Prob_{\nu}^{\infty}}\Biggl[\, \sum_{t=1}^{\infty}\beta^{t-1}\bigl(C(x^{t+1})-\rp^t(x^t-x^{t+1})\bigr) \Biggr]&&\quad\label{problem:approximating-revenue-optimization}\\
& \subjectTo &&
\quad x^{t+1} = y(x^t,\rp^t)\in[0,x^t] \quad \forall t\in \nats. \notag&& 
\end{alignat}
For each admissible control policy $y^1,y^2,\ldots\,$, it holds that
\begin{align*}
& \Expt_{\Prob_{\nu}^\infty}\Biggl[\, \sum_{t=1}^{\infty }\beta ^{t-1}\bigl( C(x^{t+1}) -\rp^{t}(x^t-x^{t+1})\bigr)\Biggr] \quad \where~x^{t+1} = y^t(x^t,\rp^t)\;\forall t\in \nats\\
\geq{}& \Expt_{\Prob_{\nu}^\infty}\Biggl[\, \sum_{t=1}^{\infty }\beta ^{t-1}\bigl( -\rp^{t}(x^1)\bigr) \Biggr] \\ 
={}& -(1-\beta)^{-1} \Expt_{\Prob_{\nu}}[\rp]x^1,
\end{align*}
the equality following from Lebesgue's dominated convergence theorem which applies since $\Prob_{\nu}$ is supported at a finite number of discrete points. Similarly,
\begin{align*}
& \Expt_{\Prob_{\nu}^\infty}\Biggl[\, \sum_{t=1}^{\infty }\beta ^{t-1}\bigl( C(x^{t+1}) -\rp^{t}(x^t-x^{t+1})\bigr) \Biggr] \quad \where~x^{t+1} = y^t(x^t,\rp^t)\;\forall t\in \nats\\
\leq{}& \Expt_{\Prob_{\nu}^\infty}\Biggl[\, \sum_{t=1}^{\infty }\beta ^{t-1} C(x^1) \Biggr] \\ 
={}& (1-\beta)^{-1} C(x^1).
\end{align*}
Since $V_{\nu}$ satisfies the principle of optimality relative to (\ref{problem:approximating-revenue-optimization}), we have
\begin{equation*}
    \inf_{\kappa\in\nats}-(1-\beta)^{-1} \Expt_{\Prob_{\kappa}}[\rp]x \leq \ExptCtgF_{\nu}(x) \leq (1-\beta)^{-1} C(x) \quad \forall x\in\reals_+.
\end{equation*}
With $\Expt_{\Prob}[\rp]$ finite, the SLLN applies to $\Expt_{\Prob_{\kappa}}[\rp]$, and thus $\Expt_{\Prob_{\kappa}}[\rp] \to \Expt_{\Prob}[\rp]$ for almost every outcome $\{{p}_1,{p}_2,\ldots\}$. Recalling that the infimum of a sequence of real numbers is finite if the sequence is convergent, 
it follows that for each $x\in\reals_+$ the value
\begin{equation*}
\inf_{\kappa\in\nats}-(1-\beta)^{-1} \Expt_{\Prob_{\kappa}}[\rp]x
\end{equation*}
is finite for almost every outcome $\{{p}_1,{p}_2,\ldots\}$. Also, $(1-\beta)^{-1}C(x)$ is real valued for each $x\in\reals_+$ since $\beta\in(0,1)$.
\end{proof}

In order to apply Theorem~\ref{theorem:infinite-horizon-consistency} we now construct a set ${\mathcal{V}\subset\lscfcns(\reals_+)}$ that contains the approximating fixed-point functions. We assume that $\Expt_{\Prob}[\rxi]$ is finite and use the pointwise bounds provided by Proposition~\ref{proposition:revenue-optimization-value-function-bounds}. Let
\begin{gather*}
    \mathcal{V} = \Bigl\{V\in\lscfcns(\reals_+) : V~\text{is a convex function} \Bigr\} \\
    \cap \Bigl\{V \in\lscfcns(\reals_+) : \inf_{\kappa\in\nats}-(1-\beta)^{-1}\Expt_{\Prob_{\kappa}}[\rp] x \leq V(x)  \leq (1-\beta)^{-1} C(x)\quad \forall x \in\reals_+ \Bigr\}
\end{gather*}
which contains the approximating fixed-point functions. As an intersection of closed sets, $\mathcal{V}$ is closed under $\awDistance$ \parencite[Propositions~4.1(i) and 4.5]{royset-wets-constrained-m-estimators}. Furthermore, since each $V\in\mathcal{V}$ satisfies $V(0)=0$, the point $z_{\ctr} = (0,0)$ is common to the epigraphs of every function in $\mathcal{V}$. It follows that $\mathcal{V}$ is bounded \parencite[Proposition~3.1]{royset-approximations-2018} and therefore compact under $\awDistance$ \parencite[Proposition~4.45, Theorem~7.58]{rockafellar-wets-variational-analysis}.

We now check that Theorem~\ref{theorem:infinite-horizon-consistency} applies to $\mathcal{V}$. Consider Assumption~\ref{assumption:infinite-horizon-regularity} and the equi-lsc of epi-convergent sequences: For a given $V\in\mathcal{V}$ and $(x,p)\in\reals_+\times\reals_+$, the value 
\begin{equation*}
b(V)(x,p) = {\inf_{0 \leq y \leq x}\Bigl\{C(y)-p(x-y) + \beta V(y)\Bigr\}}
\end{equation*}
is finite, satisfying
\begin{equation}
    \inf_{\kappa\in\nats}-\bigl(p +\beta(1-\beta)^{-1} \Expt_{\Prob_{\kappa}}[\rp]\bigr) x \leq b(V)(x,p)\leq (1-\beta)^{-1} C(x).
    \label{equation:inventory-control-pw-bounds}
\end{equation}
Moreover, the mapping $p \mapsto b(V)(x,p)$ is concave. Hence, $b(V)(x,\blank)$ is continuous and therefore measurable. Since the left- and right-hand side functions in (\ref{equation:inventory-control-pw-bounds}) are integrable, we deduce that $b(V)(x,\blank)$ is integrable as well. This shows that the operators ${\{B,B_{\nu};\, \nu\in\nats\}}$ are well defined, and from (\ref{equation:inventory-control-pw-bounds}) it can also be seen that they map $\mathcal{V}$ into itself. The equi-lsc of each epi-convergent sequence in $\mathcal{V}$ follows from convexity \parencite[Proposition~3.2(ii)]{royset-wets-constrained-m-estimators} and (\ref{equation:inventory-control-pw-bounds}). 

Now consider the conditions~(\ref{theorem:infinite-horizon-consistency-i}) and (\ref{theorem:infinite-horizon-consistency-ii}) of Theorem~\ref{theorem:infinite-horizon-consistency}: The asymptotic semi-integrability statements (\ref{theorem:infinite-horizon-consistency:equation:assymptotic-inf-integrable}) and (\ref{theorem:infinite-horizon-consistency:equation:assymptotic-sup-integrable}) follow from the left- and right-hand functions in (\ref{equation:inventory-control-pw-bounds}) being continuous and integrable. The equi-semicontinuity statements follow from $b(V)$ being a saddle function and (\ref{equation:inventory-control-pw-bounds}); see \parencite[Example~9.14]{rockafellar-wets-variational-analysis} and \parencite[Proposition~3.2(v)]{royset-wets-constrained-m-estimators}.

We thus conclude that Theorem~\ref{theorem:infinite-horizon-consistency} and Proposition~\ref{proposition:consequences-of-non-uniqueness}(\ref{proposition:consequences-of-non-uniqueness-i}) apply for almost every outcome $\{p_1,p_2,\ldots\}$. Also, if $V\in\mathcal{V}$ is a fixed point of $B$, the principle of optimality can be verified using the pointwise bounds on each of the functions in $\mathcal{V}$; see \parencite[Theorem~9.2]{stokey-lucas:recursive-methods}.

\subsection{Failure to Converge to an Optimal Policy}
\label{subsection:inventory-control-inconsistency}
We now study an example where the approximating decision rules do not converge to an optimal policy for the true control problem. Consider the problem (\ref{problem:revenue-optimization}) from Subsection~\ref{subsection:inventory-control-example}, but with $\Prob$ being the $\text{L{\'e}vy}(0,1)$ distribution having location parameter $0$ and scale parameter $1$; this is supported on all of $\reals_+$ and has a heavy right-hand tail. In fact, $\Prob$ is not integrable and the reasoning of Subsection~\ref{subsection:inventory-control-example} does not apply. 

Suppose $C(0)=0$, and take the policy that sells all inventory immediately regardless of the price, so that $y^1(x,p) = 0$ for all $(x,p)\in\reals_+\times\reals_+$. When $x^1 > 0$, such a policy is trivially optimal for (\ref{problem:revenue-optimization}), as
\begin{align*}
&  \Expt_{\Prob^\infty}\Biggl[\,\sum_{t=1}^{\infty }\beta ^{t-1}\bigl( C(x^{t+1}) -\rp^{t}(x^t-x^{t+1})\bigr) \Biggr] \quad \where~x^{t+1} = 0\;\forall t\in \nats\\
={}& \Expt_{\Prob}\bigl[-\rp^{1}x^{1}\bigr]\\ 
={}& -\infty.
\end{align*}

Following the scheme used in Subsection~\ref{subsection:inventory-control-example}, we solve the approximating fixed-point problems numerically using \texttt{SDDP.jl} \parencite{dowson_sddp.jl, dowson-policy-graph} for random outcomes of $\{p_1,p_2,\ldots\}$. Setting $\beta = 0.99$ and $C(x) = \frac{1}{2}x^2$, Figure~\ref{figure:inconsistency} graphs $y_{\nu}(1,1)$ as a function of ${\nu}$, in which $y_{\nu}(1,1) \to 1$ as $\nu\to\infty$. 
\begin{figure}[H]
    \centering
    \hspace{-1.2cm}\includegraphics[width=281pt]{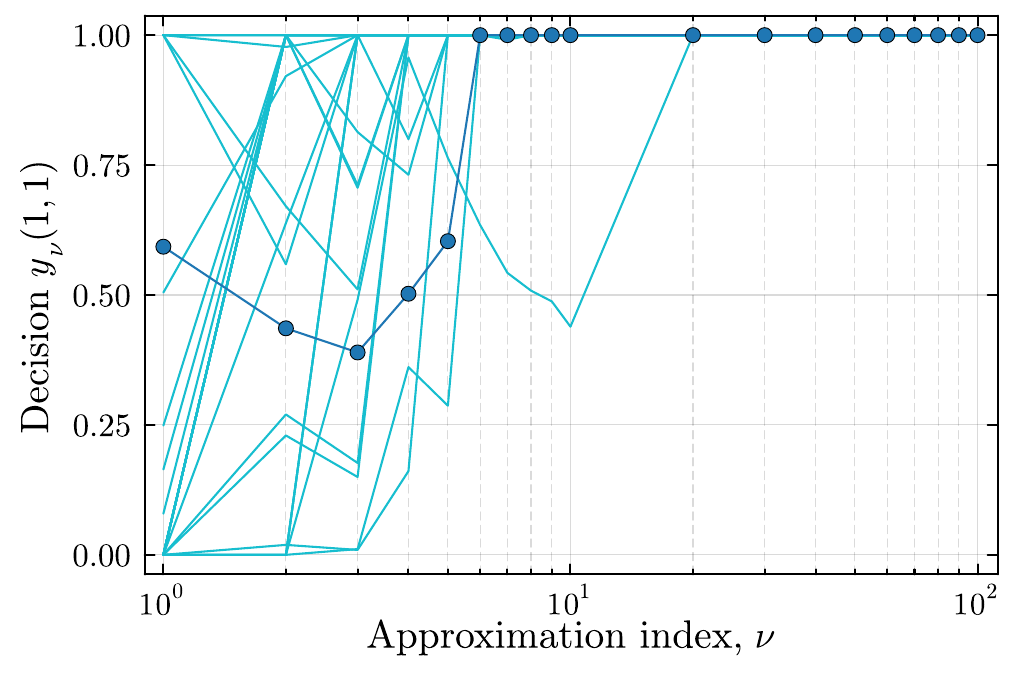}
    \caption{\textbf{Asymptotic inconsistency.} Decision for the next state $y_{\nu}(1,1)$ as a function of the number of samples $\nu$ when the true distribution is $\text{L{\'e}vy}(0,1)$. One outcome highlighted.}
    \label{figure:inconsistency}
\end{figure}

\noindent It can further be shown that $y_{\nu}(x,p) \to x$ for each $(x,p)\in\reals_+\times\reals_+$. The resulting policy is clearly suboptimal. Implementing it in (\ref{problem:revenue-optimization}) yields the objective value
\begin{align*}
& \Expt_{\Prob^\infty}\Biggl[ \,\sum_{t=1}^{\infty }\beta ^{t-1}\bigl( C(x^{t+1}) -\rp^{t}(x^t-x^{t+1})\bigr) \Biggr] \quad \where~x^{t+1} = x^t \;\forall t\in \nats\\
={}& (1-\beta)^{-1}C(x^{1})\\ 
>{}& -\infty.
\end{align*}

The failure of the approximating policy to converge to an optimal policy does not relate to a failure of the SLLN because $\Expt_{\Prob_{\nu}}[\rp] \to \Expt_{\Prob}[\rp] = \infty$ almost surely; see \parencite[Theorem~22.1]{billingsley1995probability} and its corollary. The first issue arises when attempting to construct a subset to contain the approximating fixed-point functions; the function previously used in Subsection~\ref{subsection:inventory-control-example} to specify a pointwise lower bound is not finite which does not enforce the required equi-semicontinuity properties. The second issue arises when attempting to verify the principle of optimality; even though it is satisfied for each $\nu\in\nats$ relative to an underlying approximation of the true control problem, it does not apply asymptotically as $\nu\to\infty$. 

\subsection{Stagewise-Dependent Random Variables}
\label{subsection:inventory-control-autoregressive-randomness}
Here we extend the control problem in Subsection~\ref{subsection:inventory-control-example} to model stagewise dependence of the random price by expanding the state space. We suppose that the natural logarithm of the random price follows an autoregressive process: for current log price $\ell\in\reals$ the log price in the next stage is the random variable $\alpha \ell + \rxi$, where $\alpha \in (0,1)$ and $\rxi$ is distributed according to $\Prob$ which is supported on $\reals$. In the next stage the previous log price $\ell$ becomes an additional state variable while the random variable $\rxi$ remains stagewise independent, as in the formulation (\ref{problem:sum}). The resulting fixed-point problem is to find ${V}:\reals_+\times\reals\to\reals$ such that
\begin{equation*}
    V(x,\ell)=\Expt_{\Prob}\left[\,\inf_{\substack{0\leq y\leq x\\\eta =\alpha\ell+\rxi}}\Bigl\{C(y)-\exp(\eta)(x-y)+\beta {V}(y,\eta)\Bigr\}\right] \quad \forall (x,\ell) \in \reals_+\times\reals.
\end{equation*}
The possible realizations of $\alpha\ell+\rxi$ cannot be said to lie only in some compact subset of $\reals$ when $\Prob$ has support on all of $\reals$, so we now have an unbounded state variable and hence stage-cost function. 

We do not assume that the constant $\alpha$ is available a priori and estimate it simultaneously to our approximation. Given a sequence of constants $\{\alpha_{\nu}\in(0,1)\}_{\nu\in\nats}$ with $\alpha_{\nu}\to\alpha$ and a sequence of probability distributions $\{\Prob_{\nu}\}_{\nu\in\nats}$ with $\Prob_{\nu}\weaklyto\Prob$, the approximating fixed-point problem is to find
$\ExptCtgF_{\nu}:\reals_+\times\reals\to\reals$ such that 
\begin{equation*}
    \ExptCtgF_{\nu}(x,\ell)=\Expt_{\Prob_{\nu}}\left[\,\inf_{\substack{0\leq y\leq x\\\eta =\alpha_{\nu}\ell+\rxi}}\Bigl\{C(y)-\exp(\eta)(x-y)+\beta \ExptCtgF_{\nu}(y,\eta)\Bigr\}\right]
    \quad \forall(x,\ell) \in \reals_+\times\reals.
\end{equation*}
Assuming $\Prob_{\nu}$ has compact support, future states are constrained to a compact set (consider the resulting sequences of future log prices when only the largest- or smallest-possible values of $\rxi$ under $\Prob_{\nu}$ is realized). An argument utilizing \parencite[Theorem~9.6]{stokey-lucas:recursive-methods} then shows that a unique continuous solution function $\ExptCtgF_{\nu}$ satisfying the principle of optimality relative to the underlying approximation of (\ref{problem:revenue-optimization}) exists. Similar reasoning to that in the proof of \parencite[Theorem~9.8]{stokey-lucas:recursive-methods} shows that this $\ExptCtgF_{\nu}$ is a saddle function. 

Note that in the definition of $V_{\nu}$ the constraint $\eta =\alpha_{\nu}\ell+\rxi$ varies with $\nu$. This is not explicitly permitted by the results previously developed in this paper. However, we can treat this example by viewing the constant $\alpha_{\nu}$ as a random variable that has the point-mass distribution $\pmProb_{\alpha_{\nu}}$. 
Since $\alpha_{\nu}\to\alpha$, it holds that $\pmProb_{\alpha_{\nu}}\weaklyto\pmProb_{\alpha}$, and with $\Prob_{\nu}\weaklyto\Prob$ as well, $\pmProb_{\alpha_{\nu}}\times\Prob_{\nu}\weaklyto\pmProb_{\alpha}\times\Prob$ \parencite[Theorem~2.8(ii)]{convergence-of-probability-measures:Billingsley}. So we are in the setting of Section~\ref{section:infinite-horizon-stochastic-dynamic-programming}. In what follows, under absolute-integrability and light-tailedness assumptions, we verify that the assumptions of Theorem~\ref{theorem:infinite-horizon-consistency} do indeed hold.

We start by discussing how $\alpha_{\nu}$ and $\Prob_{\nu}$ might arise in practice. Let $\rxi^1,\rxi^2,\ldots$ be random variables that are independent and identically distributed according to $\Prob$. Suppose that $\{\xi^1,\xi^2,\ldots\}$ is an outcome of this random sequence. We can view historical log prices $\ell^1,\ell^2,\ldots$ as being generated from the formula $\ell^{t+1} = \alpha\ell^t+\xi^t$ and use this to construct estimates of the unknown parameter $\alpha$. For instance, applying ordinary least squares to the first $\nu+1$ log prices yields
\begin{equation*}
\alpha_{\nu}\in\argmin_{\bar{\alpha},\bar{\mu}\in\reals} \sum_{t=1}^{\nu}\bigl(\ell^{t+1} - (\bar{\alpha}\ell^{t} + \bar{\mu})\bigr)^2.
\end{equation*}
The resulting $\alpha_{\nu}$ can then be used to construct $\Prob_{\nu}$ by computing $\xi^t_{\nu} = \ell^{t+1}-\alpha_{\nu}\ell^{t}$ for each $t \in[\nu]$ and setting $\Prob_{\nu} = \frac{1}{\nu} \sum_{i=1}^{\nu} \pmProb_{\xi^i_{\nu}}$. Assuming $\Expt_{\Prob}\bigl[\lvert\rxi\rvert\bigr]$ is finite, $\alpha_{\nu}\to\alpha$ for almost every outcome $\{{\xi}^1,{\xi}^2,\ldots\}$ \parencite[\S~3.3]{econometric-theory-and-methods:Davidson-MacKinnon}. (Alternatives to least squares apply here as well.) Due to the following result, this also implies that $\Prob_{\nu}\weaklyto\Prob$.

\begin{proposition}
In the context of this subsection, suppose that $\Expt_{\Prob}\bigl[\lvert\rxi\rvert\bigr]$ is finite. If $\alpha_{\nu} \to \alpha$ for almost every outcome  $\{{\xi}^1,{\xi}^2,\ldots\}$, then $\Prob_{\nu}\weaklyto\Prob$ as well. 
\end{proposition}

\begin{proof}
Due to the SLLN,
\begin{align}\label{equation:mean-log-price}
\frac{1}{\nu} \sum_{i=1}^{\nu} \lvert\ell^i\rvert &= \frac{1}{\nu} \sum_{i=1}^{\nu} \bigl\lvert\alpha^{i-1}\ell^1 + \alpha^{i-2}{\xi}^1+\alpha^{i-1}{\xi}^2+\cdots+{\xi}^{i-1}\bigr\rvert \notag\\
&\leq \frac{1}{\nu} \sum_{i=1}^{\nu} \bigl(\alpha^{i-1}\lvert\ell^1\rvert + \alpha^{i-2}\lvert{\xi}^1\rvert+ \alpha^{i-1}\lvert{\xi}^2\rvert+\cdots+ \lvert{\xi}^{i-1}\rvert\bigr) \notag\\
&\to (1-\alpha)^{-1} \Expt_{\Prob}\bigl[\lvert\rxi\rvert\bigr]  \in \reals
\end{align}
for almost every outcome  $\{\xi^1,\xi^2,\ldots\}$. Using the portemanteau theorem \parencite[Theorem~13.16(ii)]{probability-theory:Klenke}, we show the following:
\begin{equation*}
\Expt_{\Prob_{\nu}}\bigl[f(\rxi)\bigr]=\frac{1}{\nu}\sum_{i=1}^{\nu}f({\xi}^i_{\nu}) \to \Expt_{\Prob}\bigl[f(\rxi)\bigr]~\text{for each bounded Lipschitz-continuous}~f:\Xi\to\reals,
\end{equation*}
for almost every outcome $\{{\xi}^1,{\xi}^2,\ldots\}$. Now, for each bounded $L$-Lipschitz-continuous $f:\Xi\to\reals$,
\begin{equation}\label{equation:lipschitz-sum}
 \Biggl\lvert\frac{1}{\nu}\sum_{i=1}^{\nu}f({\xi}^i) - \frac{1}{\nu}\sum_{i=1}^{\nu}f({\xi}^i_{\nu})\Biggr\rvert \leq \frac{1}{\nu}\sum_{i=1}^{\nu} \Bigl\lvert f({\xi}^i) - f({\xi}^i_{\nu})\Bigr\rvert \leq  \frac{1}{\nu} \sum_{i=1}^{\nu} L \bigl\lvert{\xi}^i-{\xi}_{\nu}^i\bigr\rvert = \frac{1}{\nu} \sum_{i=1}^{\nu} L  \bigl\lvert(\alpha - \alpha_{\nu})\ell^i \bigr\rvert.
\end{equation}
Since $\alpha_{\nu} \to \alpha$, using (\ref{equation:mean-log-price}) we have that $\frac{1}{\nu} \sum_{i=1}^{\nu} L\lvert(\alpha - \alpha_{\nu})\ell^i \rvert\to 0$, and we deduce from (\ref{equation:lipschitz-sum}) that $\frac{1}{\nu}\sum_{i=1}^{\nu}f({\xi}^i)$ and $\frac{1}{\nu}\sum_{i=1}^{\nu}f({\xi}^i_{\nu})$ have the same limit. With $f$ bounded and Lipschitz continuous, $\frac{1}{\nu}\sum_{i=1}^{\nu}f({\xi}^i) \to \Expt_{\Prob}\bigl[f(\rxi)\bigr]$  for almost every outcome $\{{\xi}^1,{\xi}^2,\ldots\}$. The result then follows from the transitivity of limits. 
\end{proof}

As for Proposition~\ref{proposition:revenue-optimization-value-function-bounds} in Subsection~\ref{subsection:inventory-control-example}, to construct a set $\mathcal{V}\subset\lscfcns(\reals_+\times\reals)$ for which the conditions of Theorem~\ref{theorem:infinite-horizon-consistency} hold, we derive the following pointwise bound on the approximating fixed-point functions $V_{\nu}$.
\begin{proposition}\label{proposition:log-AR1-revenue-optimization-value-function-bounds}
In the context of this subsection, suppose that $\Expt_{\Prob}\bigl[\lvert\rxi\rvert\bigr]$ and $\Expt_{\Prob}\bigl[\exp(\rxi)\bigr]$ are finite. Then, for each $\nu\in\nats$ it holds that
\begin{multline*}
    \inf_{\kappa\in\nats} \sum_{t=1}^{\infty }\beta ^{t-1}\Bigl(-\exp(\alpha_{\kappa}^{t}\ell) \tprod{s=0}{t-1}\Expt_{\Prob_{\kappa}}\bigl[\exp(\alpha_{\nu}^{s}\rxi) \bigr] x\Bigr) \leq \ExptCtgF_{\nu}(x,\ell) \leq (1-\beta)^{-1} C(x)\\ \quad \forall (x,\ell)\in\reals_+\times\reals,
\end{multline*}
and these bounds are finite valued for almost every outcome $\{{\xi}^1,{\xi}^2,\ldots\}$.
\end{proposition}

\begin{proof}
For a given $\nu\in\nats$ and $(x,\ell)\in\reals_+\times\reals$, similar reasoning to that in the proof of Proposition~\ref{proposition:revenue-optimization-value-function-bounds} shows
\begin{align}
    \ExptCtgF_{\nu}(x,\ell) &\geq  \inf_{\kappa\in\nats} \Expt_{\Prob_{\kappa}^\infty}\Biggl[\, \sum_{t=1}^{\infty }\beta ^{t-1}\Bigl(-\exp\bigl(\alpha_{\kappa}^{t}\ell+ \alpha_{\kappa}^{t-1}\rxi^1 +  \alpha_{\kappa}^{t-2}\rxi^2 + \cdots + \rxi^{t}\bigr) x  \Bigr) \Biggr] \notag\\
    & =\inf_{\kappa\in\nats}\sum_{t=1}^{\infty }\beta ^{t-1}\Bigl( -\exp(\alpha_{\kappa}^{t}\ell) \tprod{s=0}{t-1}\Expt_{\Prob_{\kappa}}\bigl[\exp(\alpha_{\kappa}^{s}\rxi) \bigr] x\Bigr) \label{equation:AR1-log-price-bound}\\
    & \geq \inf_{\kappa\in\nats}\sum_{t=1}^{\infty }\beta ^{t-1}\Bigl(-\bigl(\alpha_{\kappa}^{t}\exp(\ell)+1\bigr)  \tprod{s=0}{t-1}\bigl(\alpha_{\kappa}^{s}\Expt_{\Prob_{\kappa}}\bigl[\exp(\rxi)\bigr]+1\bigr) x^1\Bigr) \notag\\
& \geq \sum_{t=1}^{\infty }  \beta ^{t-1}\Bigl(-\tsup{\kappa\in\nats} \bigl(\alpha_{\kappa}^{t}\exp(\ell)+1\bigr) \tprod{s=0}{t-1}\tsup{\kappa\in\nats}\bigl(\alpha_{\kappa}^{s}\Expt_{\Prob_{\kappa}}\bigl[ \exp(\rxi)\bigr]+1\bigr) x^1\Bigr) \label{equation:last-AR1-log-price-bound}, 
\end{align}
where the equality follows from the properties of the exponential function and the monotone convergence theorem. With $\Expt_{\Prob}\bigl[\lvert\rxi\rvert\bigr]$ and $\Expt_{\Prob}\bigl[\exp(\rxi)\bigr]$ finite, it can further be shown that (\ref{equation:last-AR1-log-price-bound}) is an infinite sum of real numbers for almost every outcome $\{{\xi}^1,{\xi}^2,\ldots\}$. 
Applying the ratio test to successive terms yields
\begin{align*}
    & \limsup_{t\to\infty} \frac
    {  \beta ^{t}\Bigl( -\sup_{\kappa\in\nats} \bigl(\alpha_{\kappa}^{t+1}\exp(\ell)+1\bigr) \prod_{s=0}^{t}\sup_{\kappa\in\nats}\bigl(\alpha_{\kappa}^{s}\Expt_{\Prob_{\kappa}}\bigl[\exp(\rxi)\bigr]+1\bigr) x^1\Bigr) } {  \beta ^{t-1}\Bigl( -\sup_{\kappa\in\nats} \bigl(\alpha_{\kappa}^{t}\exp(\ell)+1\bigr) \prod_{s=0}^{t-1}\sup_{\kappa\in\nats}\bigl(\alpha_{\kappa}^{s}\Expt_{\Prob_{\kappa}}\bigl[\exp(\rxi)\bigr]+1\bigr) x^1\Bigr) } \\
     ={}& \limsup_{t\to\infty} \beta\tsup{\kappa\in\nats}\bigl(\alpha_{\kappa}^{t}\Expt_{\Prob_{\kappa}}\bigl[\exp(\rxi)\bigr]+1\bigr) \\
    ={}& \beta \in (0,1),
\end{align*}
which shows that the sum is convergent and hence that (\ref{equation:AR1-log-price-bound}) is finite valued. Also, $(1-\beta)^{-1}C(x)$ is real valued since $\beta\in(0,1)$.
\end{proof}

\noindent Hence, assuming $\Expt_{\Prob}\bigl[\lvert\rxi\rvert\bigr]$ and $\Expt_{\Prob}\bigl[\exp(\rxi)\bigr]$ are finite, we set
\begin{equation*}
\mathcal\ExptCtgF_1 = \left\{ 
V \in \lscfcns(\reals_+ \times \reals) ~:~ 
\begin{aligned}
    & \ExptCtgF(x,\ell) \geq \inf_{\kappa \in \nats} \sum_{t=1}^{\infty} \beta^{t-1} \Bigl( -\exp(\alpha_\kappa^t \ell) 
      \tprod{s=0}{t-1}\Expt_{\Prob_\kappa}\bigl[\exp(\alpha_\nu^s \rxi) \bigr] x \Bigr) \\
    & \ExptCtgF(x,\ell) \leq (1 - \beta)^{-1} C(x) \kern 6.55em \forall (x, \ell) \in \reals_+ \times \reals
\end{aligned}
\right\},
\end{equation*}
which contains the approximating fixed-point functions. Compared to Subsection~\ref{subsection:inventory-control-example}, some further work is required to construct a subset of $\lscfcns(\reals_+\times\reals)$ satisfying the equi-semicontinuity requirements of Theorem~\ref{theorem:infinite-horizon-consistency}. We first show that the class of lsc saddle functions which are locally Lipschitz continuous with common moduli is closed under $\awDistance$.
\begin{proposition}\label{propostion:closed-classes-of-saddle-functions}
    Let $\mathcal{Z}_1\subset\reals^{r_1}$ and $\mathcal{Z}_2\subset\reals^{r_2}$ be closed convex sets. For saddle functions $\{f_{\nu};\, \nu\in\nats\}\subset\lscfcns(\mathcal{Z}_1\times\mathcal{Z}_2)$  which are locally Lipschitz continuous with common modulus $L(z_1,z_2) < \infty$ at each $(z_1,z_2)\in\mathcal{Z}_1\times\mathcal{Z}_2$, if $f_{\nu} \epito f \in \lscfcns(\mathcal{Z}_1\times\mathcal{Z}_2)$, then $f$ is a saddle function which is locally Lipschitz continuous with modulus $L(z_1,z_2)$ at each $(z_1,z_2)\in\mathcal{Z}_1\times\mathcal{Z}_2$.
\end{proposition}

\begin{proof}
Since each function in the sequence $\{f_{\nu}\}_{\nu\in\nats}$ is locally Lipschitz continuous with common moduli, the sequence $\{f_{\nu}\}_{\nu\in\nats}$ is equi-lsc on $\mathcal{Z}_1\times\mathcal{Z}_2$ \parencite[Proposition~3.2(v)]{royset-wets-constrained-m-estimators} and $f_{\nu} \epito f$ implies $f_{\nu} \pointto f$ \parencite[Theorem~7.1]{rockafellar-wets-variational-analysis}. Thus, for every $z_2\in\mathcal{Z}_2$, $\lambda\in[0,1]$, and $z_1,z_1^{\prime}\in\mathcal{Z}_1$,
    \begin{align*}
        \lambda f(z_1,z_2) + (1-\lambda) f(z_1^{\prime},z_2) &= \lim_{{\nu\to \infty}} \lambda f_{\nu}(z_1,z_2) + \lim_{{\nu\to \infty}} (1-\lambda) f_{\nu}(z_1^{\prime},z_2)\\
        &\geq \lim_{\nu\to \infty} f_{\nu}(\lambda z_1 + (1-\lambda)z_1^{\prime},z_2)\\
        &= f(\lambda z_1 + (1-\lambda)z_1^{\prime},z_2),
    \end{align*}
    where the inequality follows from the fact that $z_1\mapsto f_{\nu}(z_1,z_2)$ is convex. A similar argument shows that for every $z_1\in\mathcal{Z}_1$, $\lambda\in[0,1]$, and $z_2,z_2^{\prime}\in\mathcal{Z}_2$ the reverse inequality 
    holds. Hence, $f$ is a saddle function on $\mathcal{Z}_1\times\mathcal{Z}_2$, and \parencite[Proposition~4.4]{royset-wets-constrained-m-estimators} shows at each $(z_1,z_2)\in\mathcal{Z}_1\times\mathcal{Z}_2$ the function $f$ is locally Lipschitz continuous with modulus $L(z_1,z_2)$. 
\end{proof}

For almost every outcome $\{\xi^1,\xi^2,\ldots\}$, saddle functions in $\mathcal{V}_1$ are bounded on any open ball in the interior of their domain via the continuous pointwise bounds on the functions in $\mathcal{V}_1$. It follows that each saddle function in $\mathcal{V}_1$ is locally Lipschitz continuous with common modulus function $L:\reals_+\times\reals\to\reals$ \parencite[Example~9.14]{rockafellar-wets-variational-analysis}. Let
\begin{equation*}
\mathcal{V}_2 = \Bigl\{ V \in\lscfcns(\reals_+\times\reals) : V~\text{is an}~L\text{-locally-Lipschitz saddle function}\Bigr\},
\end{equation*}
which contains each of the approximating fixed-point functions.  Proposition~\ref{propostion:closed-classes-of-saddle-functions} shows that $\mathcal{V}_2$ is closed under $\awDistance$.

We set $\mathcal{V} = \mathcal{V}_1 \cap \mathcal{V}_2$ in order to apply Theorem~\ref{theorem:infinite-horizon-consistency}. The set $\mathcal{V}$ is closed under $\awDistance$, and since each $V\in\mathcal{V}$ satisfies $V(0,0) = 0$, the point $z_{\ctr} = (0,0,0)$ shows that $\mathcal{V}$ is bounded and therefore compact under $\awDistance$. Assumption~\ref{assumption:infinite-horizon-regularity}, the equi-lsc of epi-convergent sequences, and the conditions~(\ref{theorem:infinite-horizon-consistency-i}) and (\ref{theorem:infinite-horizon-consistency-ii}) of Theorem~\ref{theorem:infinite-horizon-consistency}, can all be verified using similar reasoning as in Subsection~\ref{subsection:inventory-control-example}. 
We thus conclude that Theorem~\ref{theorem:infinite-horizon-consistency} and Proposition~\ref{proposition:consequences-of-non-uniqueness}(\ref{proposition:consequences-of-non-uniqueness-i}) apply for almost every outcome $\{{\xi}^1,{\xi}^2,\ldots\}$. Also, if $V\in\mathcal{V}$ is a fixed point of $B$, the principle of optimality can be verified using the pointwise bounds on the functions in $\mathcal{V}$; see \parencite[Theorem~9.2]{stokey-lucas:recursive-methods}.

\section*{Acknowledgements}
The authors would like to thank Andrew Philpott for helpful discussions and comments. The second author is supported in part by the Office of Naval Research under grants N00014-24-1-2492 and N00014-24-1-2318.

\printbibliography

\end{document}

%% file: macros.tex
\usepackage[T1]{fontenc}
\usepackage{lmodern}

\usepackage[a4paper, margin=1in]{geometry}
\usepackage{setspace}
\usepackage{amssymb, amsfonts, amsthm, amsmath, mathtools}
\usepackage{nicefrac}
\usepackage{bm, bbm}
\usepackage{enumitem}
\usepackage{booktabs, multirow, multicol, makecell} 
\usepackage{comment}
\usepackage{hyperref}

\newtheoremstyle{proclaim}
  {3pt}        
  {3pt}        
  {\slshape}  
  {}        
  {\bfseries} 
  {.}       
  { }    
  {}        

\theoremstyle{proclaim}
\newtheorem{assumption}{Assumption}
\newtheorem{lemma}{Lemma}
\newtheorem{proposition}{Proposition}
\newtheorem{theorem}{Theorem}

\theoremstyle{definition}

\theoremstyle{remark}

\usepackage{appendix}

\newcommand{\blank}{\kern.15em\cdot\kern.15em}
\newcommand{\transpose}{^{\kern.075em\intercal}}
\DeclareMathOperator{\Indicator}{\mathbbm{1}}
\newcommand{\awDistance}{d\kern-.12em l} 

\newcommand{\reals}{\mathbb{R}}
\newcommand{\ball}{\mathbb{B}}
\newcommand{\nats}{\mathbb{N}}

\newcommand{\tprod}[2]{\mathop{\textstyle{\prod_{#1}^{#2}}}}

\newcommand{\tsup}[1]{\mathop{\textstyle{\sup_{#1}}}}

\newcommand{\tto}{\rightrightarrows}
\newcommand{\weaklyto}{\Rightarrow}
\newcommand{\hypoto}{\overset{\mathrm{h}}{\to}}
\newcommand{\epito}{\overset{\mathrm{e}}{\to}}
\newcommand{\pointto}{\overset{\mathrm{p}}{\to}}

\DeclareMathOperator{\fcns}{fcns}
\DeclareMathOperator{\lscfcns}{lsc-fcns}

\DeclareMathAlphabet{\mathsl}{T1}{cmr}{m}{sl}
\newcommand{\drv}{\mathsl{d}}

\newcommand{\defeq}{\coloneq}

\DeclareMathOperator{\dist}{dist}
\DeclareMathOperator{\epi}{epi}

\DeclareMathOperator*{\outerlimit}{LimOut}

\DeclareMathOperator*{\minimize}{minimize}

\let\argmin\relax
\DeclareMathOperator*{\argmin}{argmin}

\let\liminf\relax
\DeclareMathOperator*{\liminf}{liminf}
\let\limsup\relax
\DeclareMathOperator*{\limsup}{limsup}

\newcommand{\subjectTo}{\text{subject to}}

\newcommand{\where}{\text{where}}

\usepackage{microtype}

\usepackage{csquotes} 
\usepackage[british]{babel} 
\newcommand{\rxi}{\bm{\xi}}

\newcommand{\Expt}{\mathbb{E}} 
\newcommand{\Prob}{\mathbb{P}} 
\newcommand{\pmProb}{\mathbbm{1}} 
\newcommand{\given}{\kern.1em\vert\kern.1em}
\newcommand{\givenn}{\kern-.15em\bigm\vert\kern-.15em}

\usepackage{xcolor}
\usepackage{graphicx}
\usepackage{tikz}
\usetikzlibrary{positioning}
\usepackage[labelformat=simple]{subcaption}
\usepackage{float} 

\definecolor{myblue}{RGB}{0, 119, 180}

\usepackage[normalem]{ulem}

\usepackage[backend = biber,
            giveninits = false,
            maxbibnames = 10,
            doi = false, 
            isbn = false, 
            url = false, 
            eprint = false,
            style = ext-numeric, 
            ]{biblatex}
\DeclareNameAlias{sortname}{given-family} 
\renewbibmacro{in:}{} 
\DeclareDelimFormat{finalnamedelim}{\addspace\&\space}

\DeclareFieldFormat[article,inbook,incollection,inproceedings,patent,thesis,unpublished]{titlecase:title}{\MakeSentenceCase*{#1}}